\documentclass[10pt]{amsart}
\usepackage{graphicx,amscd,color,amsmath,amsfonts,amssymb,geometry,amssymb,amsthm}
\usepackage{graphicx}
\newtheorem{theorem}{Theorem}[section]
\newtheorem{proposition}[theorem]{Proposition}
\newtheorem{corollary}[theorem]{Corollary}

\newtheorem{remark}[theorem]{Remark}
\newtheorem{lemma}[theorem]{Lemma}

\newtheorem{definition}[theorem]{Definition}
\numberwithin{equation}{section}
\input epsf
\newcommand{\R}{\mathbb{R}}
\newcommand{\N}{\mathbb{N}}

\newcommand{\Z}{\mathbb{Z}}

\newcommand{\D}{\mathbb{L}}

\renewcommand{\P}{\mathcal{P}}
\newcommand{\C}{\mathbb{C}}
\newcommand{\A}{\mathcal{A}}

\newcommand{\Y}{\mathbf{Y}}
\newcommand{\F}{\mathcal{F}}

\newcommand{\X}{\mathbf{X}}
\newcommand{\T}{\mathbb{T}}

\renewcommand{\epsilon}{\varepsilon}
\renewcommand{\phi}{\varphi}

\begin{document}

\title{Interpolation of Holomorphic functions}

%\author{Richard M. Aron}

\author{Pablo Jim\'enez-Rodr\'iguez}

\begin{abstract} 
Interpolation Theory gives techniques for constructing spaces from two initial Banach spaces. We provide several conditions under which the restriction of a holomorphic map $f:X_0+X_1 \rightarrow Y_0+Y_1$ to the interpolated spaces (using some specific interpolation methods), where $f_{|X_0}:X_0 \rightarrow Y_0$ is compact, is also compact and holomorphic. 
\end{abstract}

\maketitle

%%%%%%%%%%%%%%%%%%%%%%%%%%%%%%%%%%%%%%%%%%%%%%%%%%%%%%%%%%%%%%%%%%%%%%%%%
%%%%%%%%%%%%%%%%%%%%%%%%%%%%%%%%%%%%%%%%%%%%%%%%%%%%%%%%%%%%%%%%%%%%%%%%%%%%
\section{Introduction and preliminaries}
%%%%%%%%%%%%%%%%%%%%%%%%%%%%%%%%%%%%%%%%%%%%%%%%%%%%%%%%%%%%%%%%%%%%%%%%%%%%
%%%%%%%%%%%%%%%%%%%%%%%%%%%%%%%%%%%%%%%%%%%%%%%%%%%%%%%%%%%%%%%%%%%%%%%%%%%

Interpolation theory has proved to be a very important area of study within Functional Analysis and has provided a rich variety of new techniques when studying Banach spaces in general, and $L_p$ spaces in particular. We refer the interested reader to \cite{Be} and the references therein for a complete introduction to this theory. Interpolation theory is still a very fruitful field of research, and one can consult the references \cite{ArMo,AsKruMas,CamFerManMaNa,CoDo,CoSe,Fer,HaRo, MasSin,Qiu} for a sample of recent papers published in this area. \\

Let $\mathcal{C}$ be the class of all compatible couples of Banach spaces (that is, those pairs $(X_0,X_1)$ of Banach spaces for which there exists a topological vector space $\A$ so that $X_0 \subseteq \A \supseteq X_1$ as subspaces). For $(X_0,X_1) \in \mathcal{C}$, we can endow $X_0 \cap X_1$ with the norm 
$$
\|x\|_{X_0 \cap X_1}=\max\{\|x\|_{X_0},\|x\|_{X_1}\},
$$
and $(X_0 \cap X_1,\| \cdot \|_{X_0 \cap X_1})$ is a Banach space. \\

Similarly, we may define a norm on $X_0+X_1$ by 
$$
\|x\|_{X_0+X_1}=\inf\{\|x_0\|_{X_0}+\|x_1\|_{X_1} \, : \, x=x_0+x_1 , \, x_i \in X_i \},
$$
and $(X_0+X_1,\| \cdot \|_{X_0+X_1})$ is also a Banach space (see, for instance, \cite{Be}, lemma 2.3.1, p. 24). \\

An {\bf interpolation method} (or {\bf functor}) is a function $F$ that gives, for any pair $(X_0,X_1)$ in the class $\mathcal{C}$, a Banach space $F(X_0,X_1)$ so that
$$
X_0 \cap X_1 \subseteq F(X_0,X_1) \subseteq X_0+X_1,
$$
and all the inclusions are continuous. \\

Given two compatible couples of Banach spaces $(X_0,X_1)$ and $(Y_0,Y_1)$, we will denote $f:(X_0,X_1) \rightarrow (Y_0,Y_1)$ to refer to a function $f:X_0+X_1 \rightarrow Y_0+Y_1$ so that, for $j=0,1$, $f(X_j) \subseteq Y_j$. All the theorems to appear will be interpreted differently depending on the interpolation method that will be taken into consideration. \\

Instead of studying properties of the spaces that arise when using an interpolation method, interpolation theory often tries to study the extent to which properties of linear functionals on the extremal spaces ($X_0$ and $X_1$) are maintained in the interpolated spaces (that is, the spaces that appear when the interpolation method is applied). In the second half of the 20$^{th}$ century, some authors started to consider other properties of functions, obtaining interpolation results for invertible functions or for compact linear operators, which is the focus of the results to come. The question that will interest us is to what extent compactness for a linear operator on one of the extremal spaces ($X_0$ or $X_1$) is enough to guarantee compactness of the operator on the interpolated spaces. \\

Despite the amount of work this question has motivated, it remains unsolved. Yet, it keeps attracting the attention of mathematicians, and some partial answers have been given. One of the first results on interpolation of compact operators dealt with the concrete case of $L_p$ spaces (Krasnolsel'skii, \cite{Kras}), and in 1964 the first abstract results of this kind appeared (Lions and Peetre, \cite{LiPe}, and Calder\'on, \cite{Calderon}). Cwikel studied the problem in the particular case where the interpolation method that is considered is the {\it classical real interpolation method} (see, for example, \cite{Cwi}. The interested reader can also refer to the papers by F. Cobos, D.E. Edmunds and A.J.B. Potter, \cite{CoEdPo}, or Cobos and Peetre, \cite{CoPe}). \\

With respect to the complex interpolation method (which we will introduce in Definition \ref{Cal}), Cwikel, N. Krugljak and M. Mastylo proved in 1996 (\cite{CwiKrugMas}) that the problem of whether compactness of an operator between Banach couples extends to the interpolated spaces (if the operator is compact in one of the extremal spaces) can be reduced to the case where the spaces $Y_0, \, Y_1$ and $X_0$ are reflexive and $X_0$ is compactly embedded into $Y_0$. In \cite{CwiKal}, Cwikel and Kalton completely solved the problem for the particular cases where the Banach couple $(X_0,X_1)$ is a couple of Banach lattices of measurable functions or when $X_0$ is a UMD-space (without extra conditions on the couple $(Y_0,Y_1)$). In 2010, Cwikel proved (\cite{Cwi2}) that the compactness of the operator over the interpolated spaces is guaranteed when $(Y_0,Y_1)$ is a couple of complexified Banach lattices of measurable functions on a common measure space, if one of the following conditions is satisfied (without extra conditions on the couple $(X_0,X_1)$):
\begin{enumerate}
\item $Y_0$ or $Y_1$ has absolutely continuous norm, or
\item $Y_0$ and $Y_1$ have the Fatou property.
\end{enumerate}

The results of this paper will focus not only on linear operators, but on homogeneous polynomials and on holomorphic functions (whose definition we will recall in Definition \ref{holomorphic}). \\

Before dealing with the definitions that we will need in the concrete topic of interpolation, we will need to introduce the theory of Fourier Series, in the more general setting of functions defined on Banach spaces:
\begin{definition}[\cite{CwiKal}] \label{Fourier}
Let $X$ be a Banach space (over $\C$) and $f \in L^2(\T,X)$, where $\T=\{|z|=1\}$. For $k$ in $\Z$, define the {\bf $k^{th}$ Fourier coefficient} as
$$
\hat{f}(k)={1 \over 2\pi i}\int_0^{2\pi}e^{-ikt}f(e^{it})dt.
$$
The {\bf Fourier Transform} is defined to be
$$
\begin{array}{crcl}
\mathfrak{F}:& L^2(\T,X) & \longrightarrow & \ell^2(X) \\
& f & \mapsto & \{\hat{f}(k)\}_{k=1}^{\infty},
\end{array}
$$
and it is a linear isometry.
\end{definition}

We will also make use of the theory of holomorphic functions defined on Banach spaces (over $\C$). 

\begin{definition}[\cite{ArScho}] \label{holomorphic}
Let $X$ and $Y$ be Banach spaces over $\C$ and let $f:X \rightarrow Y$ be a function. We say that $f$ is {\bf holomorphic} if, for every $x \in X$, there exists a radius $r>0$ and a sequence of continuous polynomials $\{P_mf(x):X \rightarrow Y\}_{m=0}^{\infty}$ so that $P_mf(x)$ is $m-$homogeneous and, for every $y$ with $\|x-y\|_X<r$, we can write
$$
f(y)=\sum_{m=0}^{\infty}P_mf(x)(y-x),
$$
where the convergence is uniform on every compact subset of $B^{\|\cdot\|_{X}}(x,r)$, or, equivalently, in one particular ball $B^{\|\cdot\|_X}(x,r')$. \\
The series $\sum_{m=0}^{\infty}P_mf(x)(y-x)$ is called the {\bf Taylor series} around $x$. We will also denote $\mathcal{H}(X;Y)=\{f:X \rightarrow Y \, : \, f \text{ is holomorphic}\}.$ \\

Notice that the polynomials given by the Taylor series can be calculated by means of the $m^{th}$ Fr\'echet derivative
$$
P_mf(x)={d^mf(x) \over m!}.
$$
For a holomorphic function $f:X \rightarrow Y$, we will denote the {\bf radius of convergence} of $f$ at $x \in X$ as
$$
R^c(x)=\sup\bigg\{R>0 \, : \, f(y)=\sum_{m=1}^{\infty}{d^mf(x) \over m!}(y-x) \text{ uniformly for }\|y-x\|_X \leq r<R \bigg\}.
$$
We will make use of \cite{Mu}, Theorem 12, where it is shown that $R^c(x)=R^b(x)=\sup \{r>0 \, : \, f(B(x,r)) \text{ is bounded}\}.$

\end{definition}

\begin{definition}[\cite{ArScho}]
Let $X$ and $Y$ be two Banach spaces and $f:X \rightarrow Y$ be a holomorphic mapping. We say that $f$ is {\bf compact} if, for every $x \in X$, there exists $r_x>0$ so that $f(B(x,r_x))$ is relatively compact in $Y$.
\end{definition}

In \cite{ArScho}, Proposition 3.4, the authors give a characterization of a compact holomorphic mapping in terms of the polynomials of its Taylor series. More concretely, it is shown that if $f:X \rightarrow Y$ is holomorphic, then it is compact if and only if ${d^mf(x) \over m!}$ is compact, for every $x$ in $X$ and natural number $m$. \\

Since we are dealing with polynomials, it will also be useful to recall the definition of the polar of a homogeneous polynomial:
\begin{definition} \label{polar}
Let $P:X \rightarrow Y$ be a homogeneous polynomial of degree $m$. Then there exists a unique multilinear symmetric form, the {\bf polar} of $P$, denoted by $\tilde{P}:\stackrel{m}{\overbrace{X \times \ldots \times X}} \rightarrow Y$, so that $P(x)=\tilde{P}(\stackrel{m}{\overbrace{x, \ldots,x}})$. Furthermore, the {\bf polarization identity} gives a very precise formula for recovering the polar from the polynomial:
$$
\tilde{P}(\stackrel{m}{\overbrace{x_1, \ldots ,x_m}})={1 \over 2^mm!} \sum_{\epsilon_1, \ldots, \epsilon_m= \pm 1}\epsilon_1 \cdot \ldots \cdot \epsilon_mP(\epsilon_1x_1+\ldots +\epsilon_mx_m).
$$
For a polynomial $P:X \rightarrow Y$, let us denote $\|P\|_{X \rightarrow Y}=\sup \{\|Px\|_{Y} \, : \, \|x\|_X \leq 1 \}$ and, for a multilinear $n$-form $\|L\|_{X \rightarrow Y}=\sup \{\|L(x_1, \ldots x_n)\|_{Y} \, : \, \|x_i\|_X \leq 1, \, 1 \leq i \leq n \}$. \\
In the future, we will denote, for a Banach space $X$, $B_X=\{x \in X \, : \, \|x\|_X \leq 1\}$, if there is no mistake about what norm we are using. \\
Trivially, $\|P\|_{X \rightarrow Y} \leq \|\tilde{P}\|_{X \rightarrow Y}$. Reciprocally, Martin's theorem (\cite{Di}, Proposition 1.8, p. 10) allows us to write, for a $m$-homogeneous polynomial $P$,
$$
\|\tilde{P}\|_{X \rightarrow Y} \leq {m^m \over m!}\|P\|_{X \rightarrow Y}.
$$
We will denote by $\mathcal{P}(^nX)$ the space of all homogeneous polynomials of degree $n$ from $X$ to $\C$ and $\mathcal{L}(^nX)$ the space of all multilinear $n$-symmetric forms from $X$ to $\C$. Given two compatible couples of Banach spaces $(X_0,X_1), \, (Y_0,Y_1) \in \mathcal{C}$ we will also denote by $\mathcal{P}(^n(X_0,X_1),(Y_0,Y_1))$ the space of all homogeneous polynomials $P:X_0+X_1 \rightarrow Y_0+Y_1$ so that $P$ is homogeneous of degree $n$ and $P:X_0+X_1 \rightarrow Y_0+Y_1, \,P:X_j \rightarrow Y_j \, (j=0,1)$ are continuous.

\end{definition}

Calder\'on asked in \cite{Calderon} the following analogous problem for bilinear operators: if we are given compatible couples of Banach spaces $(X_0^{(1)}, X_1^{(1)}), \, (X_0^{(2)}, X_1^{(2)})$ and $(Y_0,Y_1)$ and we consider a {\it bilinear} operator 
$$
T:(X_0^{(1)} + X_0^{(2)}) \times (X_1^{(1)} + X_1^{(2)}) \rightarrow Y_0 \cap Y_1
$$ 
that satisfies that $T(X_0^{(j)} \times X_1^{(j)}) \subseteq Y_j$ is bounded for $j=0,1$ and $T_{|X_0\times X_0}$ is compact, can we guarantee that $T_{F(X_0 \times X_0,X_1 \times X_1)}$ is compact, for a certain interpolation functor $F$? We refer to \cite{Be}, page 96, for the corresponding definitions for multilinear operators. Fern\'andez and da Silva (\cite{FerSil}) studied some particular cases under the real method, and recently in 2017, Fern\'andez-Cabrera and Mart\'inez (\cite{FerMar}) studied how the real method worked with a function parameter, and studied also the complex method. We would like to stress that, even though we will be working with polynomials, the questions concerning compact multilinear operators and compact polynomials will be analogous.\\

Let next $\D$ denote the set $\{ z \in \mathbb{C} \, : \, 1 < |z| < e\}$ and, for a compatible couple of Banach spaces $(X_0,X_1)$, define the function space $\F\{X_0,X_1\}$ as follows:
$$
\begin{aligned}
\F\{X_0,X_1\}=\big\{\phi:\overline{\D} \rightarrow X_0+X_1 \text{ so that } & \phi \in \mathcal{H}(\D,X_0+X_1), \\
& \phi:\{|z|=e^{j}\} \rightarrow X_j \text{ continuously}, \, j=0,1 \big\}.
\end{aligned}
$$ 

This space is a Banach space if given the norm 
$$
\|\phi\|_{\F\{X_0,X_1\}}=\max\left\{\max\{\|\phi(e^{it})\|_{X_0},\|\phi(e^{1+it})\|_{X_1}\} \, : \, t \in [0,2\pi] \right\}.
$$

Notice that, if $\phi \in \F\{X_0,X_1\}$, then automatically $\phi \in L^2(\T,X_0+X_1)$.

\begin{definition}[\cite{CwiKal}] \label{Cal}
Define the {\bf Calder\'on Complex Interpolation method} as the functor that associates to each value $0< \theta<1$ the intermediate space
$$
[X_0,X_1]_{\theta}=\{\phi(e^{\theta}) \, : \, \phi \in \F\{X_0,X_1\}\},
$$
which is a Banach space if endowed with the norm
$$
\|x\|_{[X_0,X_1]_{\theta}}=\inf\{\|\phi\|_{\F\{X_0,X_1\}} \, : \, x=\phi(e^{\theta}) \text{ for some }\phi \in \F\{X_0,X_1\}\}.
$$
In the following, if there is no confusion about which spaces are to be interpolated,  we will denote $\mathbf{X}_{\theta}=[X_0,X_1]_{\theta}$.
\end{definition}

\begin{definition}[\cite{CwiKal}]
For each $0<\theta<1$, define the {\bf Peetre Interpolation method} as the method which proceeds as follows: For each value $0<\theta<1$,
$$
\begin{aligned}
\left<X_0,X_1 \right>_{\theta}=\left\{x=\sum_{k \in \Z}x_k \, :\right. & \, x_k \in X_0 \cap X_1   \\
& \left.\text{ and }\sum_{k \in \Z}e^{(j-\theta)k}x_k \text{ is unconditionally convergent in }X_j, \, j=0,1 \right\}.
\end{aligned}
$$
This space is a Banach space, endowed with the norm
$$
\|x\|_{\left<X_0,X_1 \right>_{\theta}}=\inf \max_{j=0,1} \sup \bigg\|\sum_{k \in \Z}\lambda_ke^{(j-\theta)k}x_k \bigg\|_{X_j},
$$
where the supremum is taken over all complex valued $(\lambda_k)_{k=-\infty}^{\infty}$ with $|\lambda_k| \leq 1$ for all $k$, and the infimum is taken over all representations $x=\sum_{k \in \Z}x_k$ as above.

\end{definition}

With the Calder\'on complex method, we have the following classical Interpolation Theorem, due to Riesz and Thorin:

\begin{theorem}[M. Riesz, G.O. Thorin, \cite{Be}, p. 2]
Let $0 < p_i, q_i<\infty, \, 0<\theta<1$. Assume $p_0 \neq p_1$ and $q_0 \neq q_1$ and define $p, \, q$ by
$$
{1 \over p}={1-\theta \over p_0}+{\theta \over p_1}, \, {1 \over q}={1-\theta \over q_0}+{\theta \over q_1}.
$$
Assume that 
$$
T:L_{p_0}(U,d\mu)\rightarrow L_{q_0}(V,d\nu) \quad \text{and} \quad T:L_{p_1}(U,d\mu) \rightarrow L_{q_1}(V,d\nu)
$$
are linear operators bounded by $M_0$ and $M_1$, respectively. \\
Then, 
$$
T:L_{p}(U,d\mu) \rightarrow L_{q}(V,d\nu)
$$
is bounded and continuous with norm $M \leq M_0^{1-\theta}M_1^{\theta}$. \\
We remark that if $d\mu=d\nu$ we have $[L_{p_0}(A_0),L_{\infty}(A_1)]_{\theta}=L_p([A_0,A_1]_{\theta})$, for every compatible couple of Banach spaces $A_0, \, A_1$.
\end{theorem}

More generally, we will make use of the following generalization of the Interpolation Theorem for linear operators, which can be consulted in \cite{Be} (theorem 4.4.1, p. 96):

\begin{theorem} \label{multilinear}
Let $(X_0^{(j)},X_1^{(j)})_{j=1}^n, \, (Y_0,Y_1)$ be compatible couples of Banach spaces. Assume that 
$$
T:(X_0^{(1)} \cap X_1^{(1)}) \times \cdots \times (X_0^{(n)} \cap X_1^{(n)}) \rightarrow Y_0 \cap Y_1
$$
is an $n-$multilinear form and satisfies $\|T(x_1, \ldots ,x_n)\|_{Y_j} \leq M_j$, for $x_l \in B_{X_j^{(l)}}$, $1 \leq l \leq n, \, j=0,1$. \\
Then, for any $0<\theta<1$, $T$ can be uniquely extended to a multilinear mapping 
$$
T:E:=[X_0^{(1)},X_1^{(1)}]_{\theta} \times \cdots \times [X_0^{(n)} , X_1^{(n)}]_{\theta} \rightarrow [Y_0,Y_1]_{\theta},
$$
with $\|T\|_{E \rightarrow [Y_0,Y_1]_{\theta}} \leq M_0^{1-\theta}M_1^{\theta}$.
\end{theorem}

Calder\'on's method behaves very well with procedures like reiteration of the interpolation method. More concretely, the result below can be also found in \cite{Be}, Theorem 4.2.2, p. 91:

\begin{theorem} \label{Ber}
Let $0 \leq \theta \leq 1$. Then 
\begin{enumerate}
\item[a)] $X_0 \cap X_1$ is dense in $[X_0,X_1]_{\theta}$ (using the $\|\cdot\|_{\theta}$-norm).
\item[b)]
$$
[X_0,X_1]_{\theta}=\Big[\overline{X_0 \cap X_1}^{\| \cdot \|_0},X_1\Big]_{\theta}=\Big[X_0,\overline{X_0\cap X_1}^{\| \cdot \|_1}\Big]_{\theta}=\Big[\overline{X_0 \cap X_1}^{\| \cdot \|_0},\overline{X_0\cap X_1}^{\| \cdot \|_1}\Big]_{\theta}.
$$
\item[c)]The space $[X_0,X_1]_{j}$ is a closed subspace of $X_j$ for $j=0,1$, with equality of norms.
\item[d)]
$$
[X_0,X_1]_{\theta}=\Big[[X_0,X_1]_{0},[X_0,X_1]_{1}\Big]_{\theta}.
$$
\end{enumerate}
\end{theorem}

We will make use of the following result, proved by Lions and Peetre (\cite{LiPe}, ch. IV, Theorem 1.1, p. 29):

\begin{theorem} \label{Lions}
Let $X_0, \, X_1$ be Banach spaces and $0<\theta<1$. Then,
\begin{enumerate}
\item \label{part1} there exists a constant $C>0$ so that, for every $x \in X_0 \cap X_1$,
$$
\|x\|_{[X_0,X_1]_{\theta}} \leq C\|x\|_{X_0}^{1-\theta}\|x\|_{X_1}^{\theta}.
$$
\item \label{part2} There exists a constant $C'>0$ so that, for every $x \in [X_0,X_1]_{\theta}$ and $t>0$, one can find $x_0 \in X_0$ and $x_1 \in X_1$ satisfying:
$$
\begin{aligned}
x&=x_0+x_1, \\
\|x_0\|_{X_0} &\leq C't^{\theta}\|x\|_{[X_0,X_1]_{\theta}}, \\
\|x_1\|_{X_1} &\leq C't^{\theta-1}\|x\|_{[X_0,X_1]_{\theta}}.
\end{aligned}
$$
\end{enumerate}
\end{theorem}

The results of this paper are presented as follows: In section \ref{firstresults} we will give an answer to the natural question of whether a function which is holomorphic as a function between $X_0+X_1$ and $Y_0+Y_1$, and as a function between $X_j$ and $Y_j$ ($j=0,1$) is holomorphic as well when restricted to $[X_0,X_1]_{\theta}$, $0 \leq \theta \leq 1$. That is, we will proof an analogous theorem to the Riesz-Thorin Theorem, but for the more general setting of holomorphic functions instead of linear operators. This theorem will be of special importance for us if we want to reduce the study of compactness between the interpolated spaces to the study of the compactness of the polynomials that appear in the Taylor series, in virtue of Proposition 3.4 from \cite{ArScho}. \\

In the same section, we will follow the ideas suggested in \cite{CwiKal} by Cwikel and Kalton to to prove some preliminary technical lemmas. In section \ref{CwikelKalton}, we will continue with the ideas from \cite{CwiKal} to prove a theorem about compactness on the interpolated spaces, if in the domain space we use  Peetre's interpolation space and in the range space we consider Calder\'on's interpolation method. Some of the procedures Cwikel and Kalton carried out for linear operators have an analogous application for polynomials, since linearity was not especially employed in the proofs. Some other results display a very strong dependance on the linearity of the considered operator, and we will be required to reach similar conclusions through other techniques. \\

In section \ref{classic} we will focus on some classic results. We will also prove a classical polarization-like proposition (Lemma \ref{technical}), which we believe is of interest beyond Interpolation Theory.

%%%%%%%%%%%%%%%%%%%%%%%%%%%%%%%%%%%%%%%%%%%%%%%%%%%%%%%%%%%%%%%%%%%%%%%%%%%%%%%%%%%%%%%%%%%
%%%%%%%%%%%%%%%%%%%%%%%%%%%%%%%%%%%%%%%%%%%%%%%%%%%%%%%%%%%%%%%%%%%%%%%%%%%%%%%%%%%%%%%%%%%
\section{A theorem about interpolation of holomorphic functions and some supporting lemmas} \label{firstresults}
%%%%%%%%%%%%%%%%%%%%%%%%%%%%%%%%%%%%%%%%%%%%%%%%%%%%%%%%%%%%%%%%%%%%%%%%%%%%%%%%%%%%%%%%%
%%%%%%%%%%%%%%%%%%%%%%%%%%%%%%%%%%%%%%%%%%%%%%%%%%%%%%%%%%%%%%%%%%%%%%%%%%%%%%%%%%%%%%%%

Before stating the corresponding theorem for holomorphic functions, let us prove an interpolation result for continuous functions.

\begin{proposition} \label{cont}
Let $\mathbf{X}=(X_0, X_1)$ and $\mathbf{Y}=(Y_0,Y_1)$ be two couples of Banach spaces and let $f:X_0+X_1 \rightarrow Y_0+Y_1$ be a holomorphic function. Assume that $f:X_i \rightarrow Y_i$ is continuous, for $i=0,1$. Let $0 \leq \theta \leq 1$. \\
Then, $f:[X_0,X_1]_{\theta} \rightarrow [Y_0,Y_1]_{\theta}$ is continuous.
\end{proposition}

\begin{proof} Let $x \in [X_0,X_1]_{\theta}$. Then, we can find $\phi \in \F\{X_0,X_1\}$ with $x=\phi(e^{\theta})$. \\
Let $\epsilon>0$ and $t \in [0,2\pi]$. Since $f:X_0 \rightarrow Y_0$ is continuous, we can find $\delta_t^0>0$ so that, if $\|\zeta-\phi(e^{it})\|_{X_0}<\delta_t^0$, then $\|f(\zeta)-f \circ \phi(e^{it})\|_{Y_0}<{\epsilon \over 2}$. Using continuity of $\phi:\{z \in \C \, : \, |z|=1\} \rightarrow X_0$, we can find $t_1, \ldots,t_N \in [0,2\pi]$ with
$$
\{\phi(e^{it}) \, : \, t \in [0,2\pi]\} \subseteq \bigcup_{k=1}^NB\Big(\phi(e^{it_k});{\delta_{t_k}^0 \over 2}\Big).
$$
Notice that, if $\phi(e^{it}) \in B\Big(\phi(e^{it_k});{\delta_{t_k}^0 \over 2}\Big)$ and $\|\zeta-\phi(e^{it})\|_{X_0}<{\delta_{t_k}^0 \over 2}$, then $\|\zeta-\phi(e^{it_k})\|_{X_0}<\delta_{t_k}^0$ and
$$
\|f(\zeta)-f\circ \phi(e^{it})\|_{Y_0} \leq \|f(\zeta)-f\circ \phi(e^{it_k})\|_{Y_0}+\|f\circ \phi(e^{it_k})-f\circ \phi(e^{it})\|_{Y_0}<\epsilon.
$$
Analogously, we can find $t^{(1)},\ldots,t^{(L)} \in [0,2\pi]$ and $\delta^1_{t^{(1)}},\ldots,\delta^1_{t^{(L)}} >0$ so that
$$
\{\phi(e^{1+it}) \, : \, t \in [0,2\pi]\} \subseteq \bigcup_{j=1}^LB\Big(\phi(e^{1+it^{(j)}};{\delta^1_{t^{(j)}} \over 2}\Big)
$$
and, if $\phi(e^{1+it}) \in B\Big(\phi(e^{1+it^{(j)}};{\delta^1_{t^{(j)}} \over 2}\Big)$ and $\|\zeta-\phi(e^{1+it})\|_{X_1}<{\delta^1_{t^{(j)}} \over 2}$, then
$$
\|f(\zeta)-f \circ \phi(e^{1+it})\|_{Y_1}<\epsilon.
$$
Choose $\delta=\min\left\{{\delta_{t_k}^0 \over 2},{\delta_{t^{(j)}}^1 \over 2}\right\}_{k,j=1}^{N,L}>0$ and consider $\zeta \in [X_0,X_1]_{\theta}$ so that $\|\zeta-x\|_{[X_0,X_1]_{\theta}}<\delta$. \\
Then, we can find $\psi \in \F\{X_0,X_1\}$ with $\psi(e^{\theta})=\zeta-x$ and 
$$
\|\psi\|_{\F\{X_0,X_1\}}=\max_{t \in [0,2\pi]}\{\|\psi(e^{it})\|_{X_0},\|\psi(e^{1+it})\|_{X_1}\}<\delta.
$$
Also, notice that $f(\zeta)-f(x)=f \circ (\psi+\phi)(e^{\theta})-f \circ \phi(e^{\theta})=[f \circ (\psi+\phi)-f \circ \phi](e^{\theta})$. \\
Let $t \in [0,2\pi]$. Then, we can find $1 \leq t_{k_0} \leq N$ with $\phi(e^{it}) \in B\Big(\phi(e^{it_{k_0}});{\delta_{t_k}^0 \over 2}\Big)$. Now,
$$
\|(\psi+\phi)(e^{it})-\phi(e^{it})\|_{X_0}=\|\psi(e^{it})\|_{X_0}<\delta,
$$
so $\|f \circ (\psi+\phi)(e^{it})-f \circ \phi(e^{it})\|_{Y_0}<\epsilon$. \\
Analogously, $\|f \circ (\psi+\phi)(e^{1+it})-f \circ \phi(e^{1+it})\|_{Y_1}<\epsilon$.
In conclusion,
$$
\begin{aligned}
\|f(\zeta)-f(x)\|_{[Y_0,Y_1]_{\theta}}&=\inf\left\{\|g\|_{\F\{Y_0,Y_1\}} \, : \, g(e^{\theta})=f(\zeta)-f(x)\right\} \leq \|f \circ (\psi+\phi)-f \circ \phi\|_{\F\{Y_0,Y_1\}} \\
&<\epsilon
\end{aligned}
$$
for every $\|\zeta-x\|_{[X_0,X_1]_{\theta}}<\delta$, and the result follows.

\end{proof}

\begin{remark} \label{remark}
The hypothesis of $f:X_0+X_1 \rightarrow Y_0+Y_1$ being holomorphic in Theorem \ref{cont} is nothing more than a technicality to guarantee that $f:[X_0,X_1]_{\theta} \rightarrow [Y_0,Y_1]_{\theta}$ is well-defined since, by definition, an element $y \in [Y_0,Y_1]_{\theta}$ must be of the form $y=\psi(e^{\theta})$, with $\psi \in \F\{Y_0,Y_1\}$. To guarantee that $f \circ \phi \in \F\{Y_0,Y_1\}$ for $\phi \in \F\{X_0,X_1\}$, it is indeed enough to assume that $f:X_0+X_1 \rightarrow Y_0+Y_1$ is holomorphic and $f:X_i \rightarrow Y_i$ is continuous, for $i=0,1$. Precisely because of this, this theorem is the most general for continuity over the interpolated spaces that can be enunciated for Calder\'on's complex method. 
\end{remark}

Having Theorem \ref{cont} at hand, we can now focus on the question of whether the property of holomorphy can be obtained when restricted to the interpolated spaces via the Calder\'on's method, taking into account the considerations collected in Remark \ref{remark}.

\begin{theorem} \label{int_anal}
Let $\mathbf{X}=(X_0, X_1)$ and $\mathbf{Y}=(Y_0,Y_1)$ be two couples of Banach spaces and let $f:\mathbf{X} \rightarrow \mathbf{Y}$ be a function so that $f:X_0+X_1 \rightarrow Y_0+Y_1$ and $f:X_j \rightarrow Y_j$ are holomorphic $(j=0, \, 1)$. \\
Then, $f:\mathbf{X}_{\theta} \rightarrow \mathbf{Y}_{\theta}$ is holomorphic, for every $0 \leq \theta \leq 1$.
\end{theorem}

\begin{proof} Let first $x \in X_0 \cap X_1$. Then, we know there exists a sequence of homogeneous polynomials, given by $\big\{{d^mf(x) \over m!}\big\}_{m=0}^{\infty}$, so that
$$
f(\zeta)=\sum_{m=0}^{\infty}{d^mf(x) \over m!}(\zeta-x),
$$
where the convergence occurs uniformly on $B^{\| \cdot \|_{j}}(x;r)$ for $0<r<R_j(x)$, $j=0,1$. \\
Furthermore, we know
$$
{1 \over R_j(x)}=\limsup_{m \rightarrow \infty} \big\|{d^mf(x) \over m!}\big\|^{1/m}_{X_j \rightarrow Y_j}.
$$
Let us define $P_mf(x)={d^mf(x) \over m!}$ and let $\tilde{P}_mf(x)$ be as in Theorem \ref{polar}. Use the Polarization Constant and Martin's theorem (\cite{Mar}), together with Theorem \ref{multilinear}, to write
$$
\begin{aligned}
\|P_mf(x)\|_{\mathbf{X}_{\theta} \rightarrow \mathbf{Y}_{\theta}} &\leq \| \tilde{P}_mf(x) \|_{\X_{\theta} \rightarrow \mathbf{Y}_{\theta}} \leq \| \tilde{P}_mf(x) \|^{1-\theta}_{X_0 \rightarrow Y_0}\| \tilde{P}_mf(x) \|^{\theta}_{X_1 \rightarrow Y_1} \\
&\leq \left({m^m \over m!} \right) \|P_mf(x) \|^{1-\theta}_{X_0 \rightarrow Y_0}\| P_mf(x) \|^{\theta}_{X_1 \rightarrow Y_1}.
\end{aligned}
$$
Hence, taking also into account Stirling's formula,
$$
\begin{aligned}
{1 \over R_{\theta}(x)}&=\limsup_{m \rightarrow \infty} \|P_mf(x)\|^{1/m}_{\mathbf{X}_{\theta} \rightarrow \mathbf{Y}_{\theta}} \leq \limsup_{m \rightarrow \infty} \left({m^m \over m!} \right)^{1/m} \|P_mf(x) \|^{{1-\theta \over m}}_{X_0 \rightarrow Y_0}\| P_mf(x) \|^{{\theta \over m}}_{X_1 \rightarrow Y_1} \\
&\leq  \limsup_{m \rightarrow \infty} {e \over (2 \pi m)^{1/2m}} \|P_mf(x) \|^{{1-\theta \over m}}_{X_0 \rightarrow Y_0}\| P_mf(x) \|^{{\theta \over m}}_{X_1 \rightarrow Y_1}\\
&\leq e \left({1 \over R_0(x)} \right)^{1-\theta} \left({1 \over R_1(x)} \right)^{\theta},
\end{aligned}
$$

Therefore,
$$
R_{\theta}(x):=\limsup_{m \rightarrow \infty} {1 \over \|P_mf(x)\|^{1/m}_{\mathbf{X}_{\theta} \rightarrow \mathbf{Y}_{\theta}}} \geq {R_0(x)^{1-\theta}R_1(x)^{\theta} \over e}>0,
$$
and hence
$$
f(\zeta)=\sum_{m=0}^{\infty}P_mf(x)(\zeta-y)
$$
uniformly for $\zeta \in B^{\| \cdot \|_{\mathbf{X}_{\theta} \rightarrow \mathbf{Y}_{\theta}}}(x;r)$, $0<r<R_{\theta}(x)$. \\
Let now $x \in {\mathbf X}_{\theta}$. From Theorem \ref{Ber} (a)) we can find a sequence $\{x_n\}_{n=1}^{\infty} \subseteq X_0 \cap X_1$ so that 
$$
x_n \xrightarrow[n \rightarrow \infty]{\| \cdot \|_{\mathbf{X}_{\theta}}} x.
$$
Since $f:{\mathbf X}_{\theta} \rightarrow {\mathbf Y}_{\theta}$ is continuous at $x$ (because of Proposition \ref{cont}), we can find $\delta>0$ so that, if $\|x-\zeta\|_{\mathbf{X}_{\theta}}<\delta$, then $\|f(x)-f(\zeta)\|_{\mathbf{Y}_{\theta}}<1$. Let $n_0 \in \N$ so that $\|x-x_{n_0}\|_{\mathbf{X}_{\theta}}<{\delta \over 3}$. Then, $f$ is bounded for $\|\zeta-x_{n_0}\|_{\mathbf{X}_{\theta}}<{\delta \over 2}$. Since for a holomorphic function the radius of convergence of the Taylor series coincides with the radius of boundedness of the function, we can write
$$
f(\zeta)=\sum_{m=0}^{\infty}P_mf(x_{n_0})(\zeta-x_{n_0})
$$
uniformly for $\|\zeta-x_{n_0}\|_{\mathbf{X}_{\theta}}<{\delta \over 2}$. Then, the convergence of the series will also happen uniformly for $\|x-\zeta\|_{\mathbf{X}_{\theta}} \leq r$, for every $r<{\delta \over 6}$.

\end{proof}

We recall the holomorphic analogue of adjoint operator:

\begin{definition}[\cite{ArScho}]
Let $f:X \rightarrow Y$ be a holomorphic mapping. We define its {\bf adjoint} as 
$$
\begin{array}{cccl}
f^*: & Y^* & \rightarrow & \mathcal{H}(X;\C) \\
& y^* & \mapsto & y^*\circ f
\end{array}
$$
\end{definition}

Notice that in the case of $f$ being a polynomial (resp. a multilinear mapping) we obtain $f^*:Y^* \rightarrow \mathcal{P}(^nX)$ (resp. $f^*:Y^* \rightarrow \mathcal{L}(^nX)$), and in those cases we can write $\|f\|=\|f^*\|$. Also, notice (\cite{ArScho}) that $f$ is a compact mapping if and only if $f^*$ is also compact (for the case of multilinear mappings, just apply the polarization formula).\\

Taking this into account, we can prove the following lemma, analogous to a well-known classical result.

\begin{lemma}[Riemann-Lebesgue Lemma] \label{Riemann-Lebesgue} Let $X$ and $Y$ be Banach spaces, $f \in L^2(\T,X)$ and let $P:X \rightarrow Y$ be a compact polynomial. Then, $\| \widehat{Pf}(k)\|_Y \xrightarrow[|k| \rightarrow \infty]\, 0$.
\end{lemma}

\begin{proof} Given $y^* \in B_{Y^*}$ and using that $\mathfrak{F}:(L^2(\T,X)) \rightarrow \ell^2(X)$ is an isometry, we know
\begin{equation} \label{Parseval}
{1 \over 2\pi} \int_0^{2\pi}\left|y^*\big(Pf(e^{it})\big)\right|^2dt=\sum_{k \in \Z} |y^* \widehat{Pf}(k)|^2 \quad \text{(Parseval's Identity)},
\end{equation}
so that $|y^*\widehat{Pf}(k)| \xrightarrow[|k| \rightarrow \infty]\,0$. \\
Let $\epsilon>0$ and find $y_1^*, \ldots , y_N^* \in B_{Y^*}$ so that $P^*B_{Y^*} \subseteq \cup_{j=1}^NB(P^*y_j^*;\epsilon)$. Also, let $k_{\epsilon} \in \N$ so that for every $|k| \geq k_{\epsilon}$, $\max_{1 \leq j \leq N}|y_j^*\widehat{Pf}(k)|<\epsilon$. Then, if $y^* \in B_{Y^*}$, there exists $1 \leq j_0 \leq N$ such that 
$$
|y_{j_0}^* \widehat{Pf}(k)-y^*\widehat{Pf}(k)| \leq {1 \over 2\pi}\int_0^{2\pi}|(P^*y_{j_0}^*)f(e^{it})-(P^*y^*)f(e^{it})|dt<\epsilon \|f\|_{L^2}.
$$
Therefore, if $|k| \geq k_{\epsilon}$, $|y^*\widehat{Pf}(k)| \leq |y^*\widehat{Pf}(k)-y_{j_0}^*\widehat{Pf}(k)|+|y_{j_0}^* \widehat{Pf}(k)|<\epsilon(1+\|f\|_{L^2})$. \\
In conclusion, for every $y^* \in B_{Y^*}$ and $|k| \geq k_{\epsilon}$,
$$
|y^*\widehat{Pf}(k)|<\epsilon(1+\|f\|_{L^2}), \quad \text{so that} \quad \|\widehat{Pf}(k)\|_Y<\epsilon(1+\|f\|_{L^2}) \xrightarrow[\epsilon \rightarrow 0]\, 0.
$$

\end{proof}

We remark that we have used in a very concrete way the fact that $P$ is compact, and that the proof, as outlined above, does not work for an arbitrary polynomial.

\begin{lemma} \label{lem_polar}
Let $P:X \rightarrow Y$ be a compact homogeneous polynomial of degree $l$ and $\tilde{P}:\stackrel{l}{\overbrace{X \times \ldots \times X}} \rightarrow Y$ be its polar. Suppose $\{f_n^{(j)}\}_{n=1}^{\infty}$ are bounded sequences in $L^2(\T,X)$, $j=1, \ldots ,l$ and that, for some $k$ in $\N$,
$$
\lim_{n \rightarrow \infty}\int_0^{2\pi}\left| y^*\tilde{P}(f_n^{(1)}(e^{it}), \ldots ,f_n^{(l)}(e^{it}))\right|^kdt=0
$$
for every $y^*$ in $B_{Y^*}$. \\
Then,
$$
\lim_{n \rightarrow \infty}\int_0^{2\pi}\|\tilde{P}(f_n^{(1)}(e^{it}), \ldots ,f_n^{(l)}(e^{it}))\|^k_Y dt=0.
$$
\end{lemma}

\begin{proof}
Without loss of generality, assume that the sequences $\{f_n^{(j)}\}_{n=1}^{\infty}$ are bounded by $1$ and let $\{y_m^*\}_{m=1}^{\infty} \subseteq B_{Y^*}$ be such that $\{(\tilde{P}^*)(y_m^*) \}_{m=1}^{\infty}$ is dense in $(\tilde{P}^*)(B_{Y^*})$. Then, any $l-$tuples of bounded sequences $\{\zeta_n^{(1)}\}, \ldots, \{\zeta_n^{(l)}\}$ for which $\lim_{n \rightarrow \infty}y_m^*\tilde{P}(\zeta_n^{(1)}, \ldots,\zeta_n^{(l)})=0$ for every $m$ must satisfy $\lim_{n \rightarrow \infty}y^*\tilde{P}(\zeta_n^{(1)}, \ldots,\zeta_n^{(l)})=0$ for all $y^*$ in $B_{Y^*}$. Applying compactness, 
$$
\lim_{n \rightarrow \infty}\|\tilde{P}(\zeta_n^{(1)}, \ldots,\zeta_n^{(l)})\|_Y=0.
$$ 
Let us show that then we can then conclude that given $\epsilon>0$ there is a constant $C>0$ so that
$$
\|\tilde{P}(x^{(1)}, \ldots ,x^{(l)})\|^k_Y \leq \epsilon \left(\|x^{(1)}\|_X^k \cdot \ldots \cdot \|x^{(l)}\|_X^{k} \right)+C \sum_{m=1}^{\infty}{|y_m^*(\tilde{P}(x^{(1)}, \ldots,x^{(l)}))|^k \over 2^m},
$$
for every $x^{(1)}, \ldots,x^{(l)} \in X$. \\

Indeed, otherwise there exists $\epsilon>0$ such that, for every $n$ in $\N$, we can find $x_n^{(1)}, \ldots , x_n^{(l)}$ in $X$ with
$$
\|\tilde{P}(x_n^{(1)}, \ldots ,x_n^{(l)})\|^k_Y > \epsilon \left(\|x_n^{(1)}\|_X^k \cdot \ldots \cdot \|x_n^{(l)}\|_X^{k} \right)+n \sum_{m=1}^{\infty}{|y_m^*(\tilde{P}(x_n^{(1)}, \ldots,x_n^{(l)}))|^k \over 2^m}.
$$
Defining $\zeta_n^{(j)}={x_n^{(j)} \over \|x_n^{(j)}\|_X} \in B_X$ we obtain, for $m$ in $\N$, 
$$
\lim_{n \rightarrow \infty}n{|y_m^*(\tilde{P}(\zeta_n^{(1)}, \ldots,\zeta_n^{(l)}))|^k \over 2^m}+\epsilon \leq \|\tilde{P}\|^k_{X \times \ldots \times X \rightarrow Y},
$$
so that $\lim_{n \rightarrow \infty}|y_m^*(\tilde{P}(\zeta_n^{(1)}, \ldots,\zeta_n^{(l)}))|=0$ and, therefore, $\lim_{n \rightarrow \infty}\|\tilde{P}(x_n^{(1)}, \ldots ,x_n^{(l)})\|^k_Y=0$, contradicting $\|\tilde{P}(\zeta_n^{(1)}, \ldots,\zeta_n^{(l)}))\|_Y^k>\epsilon$ for every $n$. \\
Finally, given $\epsilon>0$, just write
$$
\int_0^{2\pi}\|\tilde{P}(f_n^{(1)}(e^{it}),\ldots,f_n^{(l)}(e^{it}))\|_Y^kdt \leq 2\pi \epsilon+C \sum_{m=1}^{\infty}{1\over 2^m}\int_0^{2\pi}|y_m^*\tilde{P}(f_n^{(1)}(e^{it}),\ldots,f_n^{(l)}(e^{it}))|^k,
$$
and the result follows.
\end{proof}

In the results to come, we shall use the following lemma, whose proof can be found in \cite{CwiKal}, Lemma 2-(i): 

\begin{lemma} \label{lem2} Let $X_0, \, X_1$ be a compatible couple of Banach spaces. For each $0<\theta<1$ there is a constant $C=C(\theta)$ such that, for every $\phi \in \F\{X_0,X_1\}$,
$$
\|\phi(e^{\theta})\|_{\X_{\theta}} \leq C \left[\int_0^{2\pi}\|\phi(e^{it})\|_{X_0}{dt \over 2\pi} \right]^{1-\theta}\left[\int_0^{2\pi}\|\phi(e^{1+it})\|_{X_1}{dt \over 2\pi} \right]^{\theta}.
$$
In particular, for all $x \in X_0 \cap X_1$,
$$
\|x\|_{\X_{\theta}} \leq C\|x\|_{X_0}^{1-\theta}\|x\|_{X_1}^{\theta}.
$$

\end{lemma}

\begin{definition} Let $X$ be a Banach space. For $N \in \N$ and $f \in \mathcal{H}(X,\C)$, define the functions $S_Nf$ as follows:
$$
S_Nf(z)=\sum_{|k| \leq N}\hat{f}(k)z^k+\sum_{N<|k| \leq 2N}\left(2-{|k| \over N} \right)\hat{f}(k)z^k.
$$ 
\end{definition}

By the uniform $L_1$-boundedness of the de la Vall\'ee Poussin kernels, if $(X_0,X_1)$ is a compatible couple of Banach spaces, there exists a constant $C$ such that $\|S_N\phi\|_{\F\{X_0,X_1\}} \leq C\|\phi\|_{\F\{X_0,X_1\}}$ for every $\phi \in \F\{X_0,X_1\}$, $N>0$ (for further details, check the comments in \cite{CwiKal} and the references therein). \\

The next lemma will provide some crucial tools for the proof of the main theorem in the next section:

\begin{lemma} \label{lemma3} Let $\mathbf{X}=(X_0,X_1)$ and $\mathbf{Y}=(Y_0,Y_1)$ be two compatible couples of Banach spaces. Let also $l \in \N$ and $P \in \mathcal{P}(^l\mathbf{X};\mathbf{Y})$ so that $P:X_0\rightarrow Y_0$ is compact. Then,

\begin{enumerate}

\item \label{3-a} The set $\left\{ \widehat{P\phi}(k) \, : \, \phi \in B_{\F\{X_0,X_1\}}, \, k \in \Z \right\}$ is relatively compact in $Y_0$.

\item \label{3-b} $\displaystyle \lim_{|k| \rightarrow \infty}\sup \left\{\|\widehat{P\phi}(k)\|_{Y_0}\, : \, \phi \in B_{\F\{X_0,X_1\}}\right\}=0$.

\item \label{3-c} For each $\delta>0$, there exists $L=L(\delta)$ so that, for every $\phi \in B_{\F\{X_0,X_1\}}$, 
$$
\text{card}\{k \in \Z \, : \, \| \widehat{P\phi}(k)\|_{Y_0} \geq \delta \} \leq L.
$$

\item \label{3-d} For each $0<\theta<1$, we have
$$
\lim_{|k| \rightarrow \infty} \sup \left\{\| \widehat{P\phi}(k)e^{k\theta}\|_{\mathbf{Y}_{\theta}} \, : \, \phi \in B_{\F\{X_0,X_1\}}\right\}=0.
$$

\end{enumerate}

\begin{proof} \textcolor{white}{hola} \\
\begin{enumerate}
\item First of all, notice that, again since $\mathfrak{F}$ (the Fourier transform) is a linear isometry from $L^2(\T,X_0+X_1)$ to $\ell^2(X_0+X_1)$ and $P$ is a compact polynomial, it follows that $\mathfrak{F}P$ is a compact operator and hence $\left\{ (\widehat{P\phi}(k))_{k=-\infty}^{\infty} \, : \, \phi \in B_{\F\{X_0,X_1\}}\right\} \subseteq Y_0^{\Z}$ is a relatively compact subset (considering the $\| \cdot \|_2$-norm and, as a consequence, in the sup norm). \\
Now, applying lemma \ref{Riemann-Lebesgue}, we find that in fact we have 
$$
\left\{ (\widehat{P\phi}(k))_{k=-\infty}^{\infty} \, : \, \phi \in B_{\F\{X_0,X_1\}}\right\} \subseteq c_0(Y_0).
$$
{\bf Claim:} If $W$ is a Banach space and $K \subseteq \big(c_0(W), \| \cdot \|_{\infty}\big)$ is compact, then $\{k(n) \, : \, k \in K, \, n \in \Z \}$ is a relatively compact subset of $W$. \\
Indeed, assume $\{k_j\}_{j=1}^{\infty} \subseteq K$ and $\{m_j\}_{j=1}^{\infty} \subseteq \Z$. Then, we know there exists $\{n_j\}_{j=1}^{\infty} \subseteq \N$ so that $\|k_{n_j}-k_{n_l}\|_{\infty} \xrightarrow[j,l \rightarrow \infty]\,0$. Let us show that $\{k_{n_j}(m_{n_j})\}$ converges. \\
Let $\epsilon >0$. Then, there exists $j_0$ such that, for every $j,l \geq j_0$ $\|k_{n_j}-k_{n_l}\|_{\infty}<{\epsilon \over 3}$. Let next $j_{\epsilon} \geq j_0$ so that $\|k_{n_{j_0}}(m_{n_j})\|_W<{\epsilon \over 6}$ for every $j \geq j_{\epsilon}$, and take $j,l \geq j_{\epsilon}$. Then
$$
\begin{aligned}
\|k_{n_j}(m_{n_j})-k_{n_l}(m_{n_l})\|_W &\leq \|k_{n_j}(m_{n_j})-k_{n_{j_0}}(m_{n_j})\|_W+\|k_{n_{j_0}}(m_{n_l})\|_W \\
& \quad+\|k_{n_{j_0}}(m_{n_l})\|_W+\|k_{n_l}(m_{n_l})-k_{n_{j_0}}(m_{n_l})\|_W \\
&<\epsilon.
\end{aligned}
$$

\item Let us prove a more general result, namely that if $\{x_j \}_{j \in J} \subseteq c_0(Y_0)$ is a compact subset, then
$$
\lim_{|n| \rightarrow \infty} \sup_{j \in J}x_j(n)=0.
$$
Indeed, assume otherwise that we can find $\epsilon>0$ such that, for every natural number $n$ there exists $N_n \geq n$ and $j_{N_n} \in J$ so that $\|x_{j_{N_n}}(N_n)\|_{Y_0} \geq \epsilon$. Without loss of generality, we can assume that $\{N_n \}_{n=1}^{\infty}$ is an increasing sequence. \\
Now, by compactness, we can find $j_0 \in J$ and a subsequence (which, to simplify the notation, we will still denote by $\{x_{j_{N_n}}\}$) so that
$$
x_{j_{N_n}} \xrightarrow[n \rightarrow \infty]{\| \cdot \|_{\infty}}x_{j_0}.
$$
Hence, given ${\epsilon \over 2}>0$, there exists a natural number $n_{\epsilon}$ so that, for every $n \geq n_{\epsilon}$ we can guarantee $\|x_{j_0}-x_{j_{N_n}}\|_{\infty}<{\epsilon \over 2}$. In other words, for every $n \geq n_{\epsilon}$ and every integer $k$ we have
$$
\|x_{j_0}(k)-x_{j_{N_n}}(k)\|_{Y_0}<{\epsilon \over 2}.
$$
In particular, for every $n \geq n_{\epsilon}$ we have
$$
\|x_{j_0}(N_n)\|_{Y_0} \geq {\epsilon \over 2},
$$
which contradicts $x_{j_0} \in c_0(Y_0)$. The argument for the case when $n \rightarrow -\infty$ follows in the same way.

\item Using compactness of $P^*$ we can find $y_1^*, \ldots, y_N^*$ in $B_{Y_0^*}$ so that
$$
P^*\overline{B}_{Y^*_0} \subseteq \bigcup_{j=1}^N B\left(P^*y_j^*;{\delta \over 2}\right).
$$
Therefore, if $y^* \in B_{Y^*_0}$, there exists $1 \leq j_0 \leq N$ so that, for every $x \in B_{X_0}$, 
$$
|(P^*y_{j_0}^*-P^*y^*)(x)|<{\delta \over 2}.
$$
Hence, for every $k \in \Z$,
$$
\begin{aligned}
|y^*\widehat{P\phi}(k)|&=\bigg|{1 \over 2\pi i}\int_0^{2\pi}y^*P(\phi(e^{it}))e^{-ikt}dt \bigg| \\
& \leq \bigg|{1 \over 2\pi i}\int_0^{2\pi}\left[y^*(P\phi(e^{it}))-y_{j_0}^*(P\phi(e^{it})) \right]e^{-ikt}dt \bigg| \\
& \quad +\bigg|{1 \over 2\pi i}\int_0^{2\pi} y_{j_0}^*(P\phi(e^{it})) e^{-ikt}dt \bigg| \\
& \leq {\delta \over 2}+\max_{1 \leq j \leq N}|y_{j}^*\widehat{P\phi}(k)|.
\end{aligned}
$$
From this
$$
\|\widehat{P\phi}(k)\|_{Y_0}=\sup_{y^* \in B_{Y_0^*}}|y^*\widehat{P\phi}(k)| \leq {\delta \over 2}+\max_{1 \leq j \leq N}|y_{j}^*\widehat{P\phi}(k)|.
$$
Let next $\phi \in B_{\F\{X_0,X_1\}}$ and define $A^{(\phi)}_{\delta}=\{k \in \Z \, : \, \| \widehat{P\phi}(k)\|_{Y_0} \geq \delta \}$. Then, for each $k \in A^{(\phi)}_{\delta}$,
$$
\sum_{j=1}^N |y_j^* \widehat{P\phi}(k)|^2 \geq {1 \over 4}\delta^2
$$
Apply Parseval's Identity \eqref{Parseval} to obtain
$$
\begin{aligned}
{1 \over 4} \delta^2 \text{card}(A^{(\phi)}_{\delta}) &\leq \sum_{j=1}^N \sum_{k \in A^{(\phi)}_{\delta}}|y_j^* \widehat{P\phi}(k)|^2 \leq \sum_{j=1}^N \sum_{k \in \Z}|y_j^* \widehat{P\phi}(k)|^2 \\
& = \sum_{j=1}^N {1 \over 2\pi}\int_0^{2\pi} |y_j^*(P\phi(e^{it}))|^2dt \leq N\|P\|^2,
\end{aligned}
$$
for every $\phi \in B_{\F\{X_0,X_1\}}$.

\item First, compute
$$
\|\widehat{P\phi}(k)\|_{\mathbf{Y}_{\theta}} \leq C\|\widehat{P\phi}(k)\|_{Y_0}^{1-\theta}\|\widehat{P\phi}(k)\|^{\theta}_{Y_1} \leq C\|\widehat{P\phi}(k)\|_{Y_0}^{1-\theta}e^{-k\theta},
$$
using Lemma \ref{lem2}  and the change $u=i+t$, for every $k \in \Z$. \\
Therefore, using part \ref{3-b}, we can conclude $\|\widehat{P\phi}(k)e^{k\theta}\|_{\mathbf{Y}_{\theta}} \xrightarrow[|k| \rightarrow \infty]\,0$. \\
Assume next (in order to simplify the notation) that $\|P\|_{X_0 \rightarrow Y_0}\leq 1$, so that $\|\widehat{P\phi}(k_n)\|_{Y_0} \leq 1$ for every $\phi \in B_{\F\{X_0,X_1\}}$ and $n \in \N$. Assume also that there exists $\delta>0$, $\{k_n\} \searrow -\infty$ and $\{\phi_n\} \subseteq B_{\F\{X_0,X_1\}}$ so that $k_n < 2k_{n-1}$ and $\|\widehat{P\phi_n}(k_n)e^{k_n\theta}\|_{\mathbf{Y}_{\theta}} \geq \delta$. \\
Now, given $n$ and $\epsilon>0$, we can use part \ref{3-a} to find $m, \, p \in \N$ so that $m>p \geq n$ and $\|\widehat{P\phi}_m(k_m)-\widehat{P\phi}_p(k_p)\|_{Y_0}<\epsilon$. \\
Also, $\|\widehat{P\phi}_m(k_m)e^{k_m}\|_{Y_1} \leq 1$ and $\|\widehat{P\phi}_p(k_p)e^{k_m}\|_{Y_1} \leq e^{k_m-k_p} \leq 1$. Then,
$$
\begin{aligned}
\|\widehat{P\phi}_m(k_m)-\widehat{P\phi}_p(k_p)\|_{\mathbf{Y}_{\theta}} &\leq C\| \widehat{P\phi}_m(k_m)-\widehat{P\phi}_p(k_p)\|_{Y_0}^{1-\theta}\left[\|\widehat{P\phi}_m(k_m)\|_{Y_1}+\|\widehat{P\phi}_p(k_p)\|_{Y_1} \right]^{\theta} \\
& \leq 2^{\theta}C\epsilon^{1-\theta}e^{-\theta k_m}.
\end{aligned}
$$
Therefore,
$$
\begin{aligned}
\|\widehat{P\phi}_m(k_m)e^{k_m\theta}\|_{\mathbf{Y}_{\theta}} & \leq \| [\widehat{P\phi}_m(k_m)-\widehat{P\phi}_p(k_p)]e^{k_m\theta}\|_{\mathbf{Y}_{\theta}}+\|\widehat{P\phi}_p(k_p)e^{k_m\theta}\|_{\mathbf{Y}_{\theta}} \\
& \leq 2^{\theta}C\epsilon^{1-\theta}+e^{k_m\theta}C\|\widehat{P\phi}_p(k_p)\|_{Y_0}^{1-\theta}\|\widehat{P\phi}_p(k_p)\|_{Y_1}^{\theta} \\
& \leq 2^{\theta}C\epsilon^{1-\theta}+Ce^{\theta(k_m-k_p)} \\
& \leq 2^{\theta}C\epsilon^{1-\theta}+Ce^{\theta k_n} \xrightarrow[n \rightarrow \infty]\,2^{\theta}C\epsilon^{1-\theta} \xrightarrow[\epsilon \rightarrow 0]\,0,
\end{aligned}
$$
which is a contradiction to the assumption $\|\widehat{P\phi}_n(k_n)e^{k_n\theta}\|_{\mathbf{Y}_{\theta}} \geq \delta>0$.

\end{enumerate}

\end{proof}

\end{lemma}

\begin{lemma} \label{4} Let $\mathbf{X}=(X_0,X_1)$ and $\mathbf{Y}=(Y_0,Y_1)$ be two compatible couples of Banach spaces and let $l \in \N$ and $P \in \mathcal{P}(^l\mathbf{X};\mathbf{Y})$ so that $P:X_0 \rightarrow Y_0$ is compact. For $0<\theta<1$ and $N \in \N$, the set $\{S_N \big(P\phi\big)(e^{\theta}) \, : \, \phi \in B_{\F\{X_0,X_1\}} \}$ is relatively compact in $\mathbf{Y}_{\theta}$.

\end{lemma}

\begin{proof} Suppose $\{\phi_n\}_{n=1}^{\infty} \subseteq B_{\F\{X_0,X_1\}}$. By lemma \ref{lemma3}, part \ref{3-a}, we can pass to a subsequence $\{ \psi_n\}$ such that, for every $|k| \leq 2N$,
$$
\| \widehat{P\psi}_n(k)-\widehat{P\psi}_{n+1}(k)\|_{Y_0}<{1 \over 2^n},
$$
for all $n \in \N$. Then,
$$
\|S_N\big(P\psi_n\big)(z)-S_N\big(P\psi_{n+1}\big)(z)\|_{Y_0} \leq {4N+1 \over 2^n},
$$
for $|z|=1$ and all natural numbers $n$. Also, for $|z|=e$, we have $\|S_N\big(P\psi_n\big)(z)-S_N\big(P\psi_{n+1}\big)(z)\|_{Y_1} \leq C_1$, for some suitable constant $C_1>0$ and natural number $n$. \\
Thus, by lemma \ref{lem2},
$$
\|S_N\big(P\psi_n\big)(e^{\theta})-S_N\big(P\psi_{n+1}\big)(e^{\theta})\|_{\mathbf{Y}_{\theta}} \leq {CC_1(4N+1)^{1-\theta} \over 2^{n(1-\theta)}},
$$
so $\{ S_N\big(P\psi_n\big)(e^{\theta})\}_{n=1}^{\infty}$ is convergent.

\end{proof}

\begin{lemma} \label{5} Let $\mathbf{X}=(X_0,X_1)$ and $\mathbf{Y}=(Y_0,Y_1)$ be two compatible couples of Banach spaces and let $\mho$ be a subset of $B_{\F\{X_0,X_1\}}$. Choose $0<\theta<1$ and define $\mho_{\theta}=\{\phi(e^{\theta}) \, : \, \phi \in \mho \}$. Assume that every sequence $\{\phi_n\} \subseteq \mho$ satisfies
$$
\lim_{n \rightarrow \infty} \|P\phi_n(e^{\theta})-S_n\big(P\phi_n\big)(e^{\theta})\|_{\mathbf{Y}_{\theta}}=0.
$$
Then, $P(\mho_{\theta})$ is relatively compact in $\mathbf{Y}_{\theta}$.

\end{lemma}

\begin{proof} 
First of all, we remark that the hypotheses imply that $\lim_{n \rightarrow \infty}\|P\phi(e^{\theta})-S_n\big(P\phi\big)(e^{\theta})\|_{\mathbf{Y}_{\theta}}=0$ uniformly for $\phi \in \mho$. Indeed, otherwise we can find $\epsilon>0$ and two subsequences, $\{N_n\}_{n=1}^{\infty} \subseteq \N$ and $\{\phi_{N_n}\}_{n=1}^{\infty} \subseteq \mho$, such that $\|P\phi_{N_n}(e^{\theta})-S_{N_n}\big(P\phi_{N_n}\big)(e^{\theta})\|_{\mathbf{Y}_{\theta}} \geq \epsilon$, which contradicts the assumptions. \\
Let now $\{\phi_n\} \subseteq \mho$. From Lemma \ref{4} we know that $\{S_1\big(P\phi_{n}\big)(e^{\theta})\}$ contains a subsequence $\{S_1\big(P\phi_{n_{k,1}}\big)(e^{\theta})\}_{k=1}^{\infty}$ which is convergent in $\mathbf{Y}_{\theta}$. Inductively, assume we have obtained a subsequence $\{\phi_{n_{k,j}}\}_{k=1}^{\infty}$ so that $\{S_l\big(P\phi_{n_{k,j}}\big)(e^{\theta}) \}_{j=1}^{\infty}$ converges, for $1 \leq l \leq j$. Then, using again Lemma \ref{4}, $\{S_{j+1}\big(P\phi_{n_{k,j}}\big)(e^{\theta}) \}_{k=1}^{\infty}$ contains a subsequence $\{S_{j+1}\big(P\phi_{n_{k,j+1}}\big)(e^{\theta})\}_{k=1}^{\infty}$ which is convergent in $\mathbf{Y}_{\theta}$. \\
Choose then $\{\phi_{n_{k,k}}\}_{k=1}^{\infty}$ and let us show that $\{P\phi_{n_{k,k}}(e^{\theta})\}$ is a Cauchy sequence. Indeed, let $\epsilon>0$. Then, we know there exists $N \in \N$ so that, for every $n \geq N$, 
$$
\|P\phi_{n_{k,k}}(e^{\theta})-S_n\big(P\phi_{n_{k,k}}\big)(e^{\theta})\|_{\mathbf{Y}_{\theta}}<\epsilon/3,
$$
for every $k \geq 1$. \\
Then, since $\{S_N\big(P\phi_{n_{k,N}}\big)(e^{\theta}) \}_{k=1}^{\infty}$ is convergent, there exists $\tilde{N}\geq N$ so that, if $l,m \geq \tilde{N}$, 
$$
\|S_N\big(P\phi_{n_{l,N}}\big)(e^{\theta})-S_N\big(P\phi_{n_{m,N}}\big)(e^{\theta})\|_{\mathbf{Y}_{\theta}}<\epsilon/3.
$$
Let $N_{\epsilon}=\tilde{N}$ and $i,j \geq N_{\epsilon}$. Then, $n_{i,i}=n_{l,N}$ and $n_{j,j}=n_{m,N}$ for some $l,m \geq \tilde{N}$. Hence,
$$
\begin{aligned}
\|P\phi_{n_{i,i}}(e^{\theta})-P\phi_{n_{j,j}}(e^{\theta})\|_{\mathbf{Y}_{\theta}} & \leq \|P\phi_{n_{i,i}}(e^{\theta})-S_N\big(P\phi_{n_{i,i}}\big)(e^{\theta})\|_{\mathbf{Y}_{\theta}}+\|S_N\big(P\phi_{n_{i,i}}\big)(e^{\theta})-S_N\big(P\phi_{n_{j,j}}\big)(e^{\theta})\|_{\mathbf{Y}_{\theta}} \\
& \quad+\|P\phi_{n_{j,j}}(e^{\theta})-S_N\big(P\phi_{n_{j,j}}\big)(e^{\theta})\|_{\mathbf{Y}_{\theta}} \\
&<\epsilon.
\end{aligned}
$$
\end{proof}

%%%%%%%%%%%%%%%%%%%%%%%%%%%%%%%%%%%%%%%%%%%%%%%%%%%%%%%%%%%%%%%%%%%%%%%%%%%%%%%%%%%%%%%%%%%%%%%%%%%%%%%%%%%%%%%%%%%%%%%%%%%%%%%%%%%%%%%%%%%%%%%%%%%%%%%%%%%%%%%%%%%%%%%%%%%%%%%%%%%%%%%%%%%%%%%%%%%%%%%%%%%%%%%%%%%%%%%%%%%%%%
\section{An interpolation result for compact holomorphic functions, by the methods of Cwikel and Kalton.} \label{CwikelKalton}
%%%%%%%%%%%%%%%%%%%%%%%%%%%%%%%%%%%%%%%%%%%%%%%%%%%%%%%%%%%%%%%%%%%%%%%%%%%%%%%%%%%%%%%%%%%%%%%%%%%%%%%%%%%%%%%%%%%%%%%%%%%%%%%%%%%%%%%%%%%%%%%%%%%%%%%%%%%%%%%%%%%%%%%%%%%%%%%%%%%%%%%%%%%%%%%%%%%%%%%%%%%%%%%%%%%%%%%%%

The main result appears as a corollary to Theorem \ref{theo}, which is itself supported by the lemmas presented in the second half of Section \ref{firstresults}.

\begin{theorem} \label{theo} Let $\X=(X_0,X_1)$ and $\Y=(Y_0,Y_1)$ be compatible couples of Banach spaces. Let $l \in \N$ and $P\in\mathcal{P}(^l\mathbf{X};\mathbf{Y})$ so that $P:X_0 \rightarrow Y_0$ is compact and let $\mho$ be the subset of $B_{\F\{X_0,X_1\}}$ consisting of those elements $\phi$ for which the series $\sum_{k \in \Z}e^{jk}\hat{\phi}(k)$ converges unconditionally in $X_j$ and $\left\|\sum_{k \in \Z} \lambda_k e^{jk} \hat{\phi}(k)\right\|_{X_j}<1$ for $j=0, \, 1$ and for every sequence of complex scalars $\{\lambda_k\}$ so that $|\lambda_k|<1$ for all $k$. \\
Then, $P(\mho_{\theta})$ is relatively compact in $\mathbf{Y}_{\theta}$ for every $0<\theta<1$.

\end{theorem}

\begin{proof} Let $0<\theta<1$ and assume (without loss of generality and to simplify the calculations) that $\|P\|_{\mathbf{X}_{\theta} \rightarrow \mathbf{Y}_{\theta}} \leq 1$. Let $\{\phi_n \} \subseteq \mho$. It will suffice to show $\lim_{n \rightarrow \infty}\|P\phi_n(e^{\theta})\|_{\mathbf{Y}_{\theta}}=0$. Indeed, using Lemma \ref{4} we can conclude that
$$
\lim_{n \rightarrow \infty}\|P\phi_n(e^{\theta})-S_n\big(P\phi_n\big)(e^{\theta})\|_{\mathbf{Y}_{\theta}}=0,
$$
so we would be in the situation of Lemma \ref{5} and the result would follow.
To simplify the notation, assume without loss of generality that $\widehat{P\phi}_n(k)=0$ for $n \in \N$ and $|k| \leq n$. \\
For any $N \in \N$ let us pick a subset $A_n(N) \subseteq \Z$ so that $\text{card}(A_n(N))=N$ and $\|\widehat{P\phi}_n(k) \|_{\mathbf{Y}_{\theta}} \leq \| \widehat{P\phi}_n(l)\|_{\mathbf{Y}_{\theta}}$ whenever $k \notin A_n(N)$ and $l \in A_n(N)$. We may use Lemma \ref{lemma3}, part \ref{3-d} to see that for any fixed $N$ we must have
$$
\lim_{n \rightarrow \infty}\bigg\| \sum_{k \in A_n(N)} \widehat{P\phi}_n(k)e^{k\theta} \bigg\|_{\mathbf{Y}_{\theta}}=0.
$$
It is therefore possible to pick a non-decreasing sequence of integers $N_n$ with $\lim_{n \rightarrow \infty}N_n=\infty$ so that
$$
\lim_{n \rightarrow \infty}\bigg\| \sum_{k \in A_n(N_n)} \widehat{P\phi}_n(k)e^{k\theta} \bigg\|_{\mathbf{Y}_{\theta}}=0.
$$
We need to deal now with $\sum_{k \notin A_n(N_n)} \widehat{P\phi}_n(k)z^k$. Let $b_n=\sup_{k \in \Z}\| \widehat{P\phi}_n(k)\|_{\mathbf{Y}_{\theta}}$. Then, by Lemma \ref{lemma3}, part \ref{3-c}, we have $\lim_{n \rightarrow \infty}b_n=0$. \\
\noindent
Let $y^* \in B_{Y_0^*} \cap B_{Y_1^*}$. Then,
\begin{equation} \label{convergence}
\begin{aligned}
\int_0^{2\pi}\bigg|y^*\sum_{k \notin A_n(N_n)}\widehat{P\phi}_n(k)e^{ikt}\bigg|^2{dt \over 2\pi} &= \sum_{k \notin A_n(N_n)}|y^* \widehat{P\phi}_n(k)|^2 \\
& \leq b_n \sum_{k \in \Z}|y^* \widehat{P\phi}_n(k)|=b_n\sup_{|\lambda_k|\leq 1}\bigg|\sum_{k \in \Z}\lambda_ky^* \widehat{P\phi}_n(k)\bigg|.
\end{aligned}
\end{equation}
\noindent
Now, notice that, if $T$ is a bilinear form (for the sake of simplification of the notation, we will assume $\|T\| \leq 1$),
\begin{equation} \label{key}
\begin{aligned}
\sum_{j_1 \in \Z} \sum_{j_2 \in \Z}|y^*T(\hat{\phi}_n(j_1),\hat{\phi}_n(j_2))|&=\sum_{j_1 \in \Z} \bigg| \sum_{j_2 \in \Z} \sup_{|\lambda_{j_2}| \leq 1}\lambda_{j_2}y^*T(\hat{\phi}_n(j_1),\hat{\phi}_n(j_2)) \bigg| \\
&=\sum_{j_1 \in \Z}\bigg|\sup_{|\lambda_{j_2}|\leq 1}y^*T\left(\hat{\phi}_n(j_1),\sum_{j_2 \in \Z}\lambda_{j_2}\hat{\phi}_n(j_2)\right)\bigg| \\
&=\sup_{\stackrel{|\lambda_{j_1}| \leq 1}{|\lambda_{j_2}| \leq 1}} \bigg|\sum_{j_1 \in \Z}\lambda_{j_1} y^* T\left(\hat{\phi}_n(j_1), \sum_{j_2 \in \Z}\lambda_{j_2}\hat{\phi}_n(j_2)\right)\bigg| \\
&=\sup_{\stackrel{|\lambda_{j_1}| \leq 1}{|\lambda_{j_2}| \leq 1}} \bigg| y^* T\left(\sum_{j_1 \in \Z}\lambda_{j_1}\hat{\phi}_n(j_1), \sum_{j_2 \in \Z}\lambda_{j_2}\hat{\phi}_n(j_2)\right)\bigg|<\infty,
\end{aligned}
\end{equation}
taking into consideration the fact that $\phi_n \in \mho$. We remark that the steps for proving equation \eqref{key} can be followed in order to obtain the analogous result for an $s$-multilinear form $T$,
$$
\sum_{j_1 \in \Z} \ldots \sum_{j_s \in \Z}|y^*T(\hat{\phi}_n(j_1),\ldots,\hat{\phi}_n(j_s))|<\infty,
$$
making the corresponding changes. \\
Therefore, we can deduce that
\begin{equation}
\begin{aligned}
y^*\widehat{P\phi}_n(k)&=y^*{1 \over 2\pi i}\int_0^{2\pi}\tilde{P}\big(\phi_n(e^{it}), \ldots,\phi_n(e^{it})) \big)e^{-ikt}\,dt \\
&=y^*{1 \over 2\pi i}\int_0^{2\pi}\tilde{P}\bigg(\sum_{j_1 \in \Z}\hat{\phi}_n(j_1)e^{ij_1t}, \ldots,\sum_{j_l \in \Z}\hat{\phi}_n(j_l)e^{ij_lt}e^{-ikt} \bigg) \, dt \\
&\stackrel{\eqref{key}}{=}\sum_{j_1, \ldots j_l \in \Z}y^{*}\tilde{P}\big(\hat{\phi}(j_1), \ldots,\hat{\phi}(j_l)\big){1 \over 2\pi i}\int_0^{2\pi}e^{i(j_1+ \ldots+j_l-k)t}\,dt \\
&= \sum_{j_1+ \ldots +j_l=k}y^* \tilde{P}(\hat{\phi_n}(j_1), \ldots , \hat{\phi_n}(j_l)).
\end{aligned}
\end{equation}
\noindent
On the other hand,
$$
\begin{aligned}
\sup_{|\lambda_k|\leq 1}\bigg|\sum_{k \in \Z}\lambda_ky^* \widehat{P\phi}_n(k)\bigg|&=\sup_{|\lambda_k\leq 1}\bigg|\sum_{k \in \Z}\lambda_k\sum_{j_1+ \ldots +j_l=k}y^* \tilde{P}(\hat{\phi_n}(j_1), \ldots , \hat{\phi_n}(j_l))\bigg| \\ 
& \leq \sup_{|\lambda_{j_1}|, \ldots , |\lambda_{j_l}| \leq 1}\bigg|y^* \sum_{j_1, \ldots , j_l \in \Z} \tilde{P}(\lambda_{j_1} \hat{\phi_n}(j_1), \ldots , \lambda_{j_l} \hat{\phi_n}(j_l)) \bigg| \\
&=\sup \bigg|y^* \tilde{P} \bigg(\sum_{j_1 \in \Z}\lambda_{j_1}\hat{\phi_n}(j_1), \ldots , \sum_{j_l \in \Z}\lambda_{j_l}\hat{\phi_n}(j_l)\bigg)\bigg| \\
&\leq {l^l \over l!}\|P\|_{\mathbf{X}_{\theta} \rightarrow \mathbf{Y}_{\theta}} \leq {l^l \over l!},
\end{aligned}
$$
since $\|P\|_{\mathbf{X}_{\theta} \rightarrow \mathbf{Y}_{\theta}} \leq 1$. Hence, keeping in mind that $b_n=\sup_{k \in \Z}\| \widehat{P\phi}_n(k)\|_{\mathbf{Y}_{\theta}} \xrightarrow[n \rightarrow \infty]\,0$, we can conclude from \eqref{convergence} that
$$
\int_0^{2\pi}\bigg|y^*\sum_{k \notin A_n(N_n)}\widehat{P\phi}_n(k)e^{ikt}\bigg|^2{dt \over 2\pi} \xrightarrow[n \rightarrow \infty]\,0
$$
for every $y^* \in B_{Y_0^*}$. \\
Next, if we call $\psi_n(t)=\sum_{k \notin A_n(N_n)}\hat{\phi}_n(k)$, we can deduce that
$$
\sum_{k \notin \A_n(N_n)} \widehat{P\phi}_n(k)e^{ikt}=\tilde{P}\big(\stackrel{l-1}{\overbrace{\phi_n(t), \ldots , \phi_n(t)}},\psi_n(t)\big).
$$
To justify this last equality, we will give the details for the case where the degree of the polynomial is $2$. The reader shall keep in mind that the general case follows the same steps, with the appropriate adaptation.
$$
\begin{aligned}
\sum_{k \notin A_n(N_n)}y^* \hat{P\phi_n}(k)e^{ikt}&=\sum_{k \notin A_n(N_n)}y^* {1 \over 2\pi}\int_0^{2\pi}P\phi_n(e^{is})e^{-iks}dse^{ikt} \\
&=\sum_{k \notin A_n(N_n)}y^* {1 \over 2\pi}\int_0^{2\pi}\sum_{j_1 \in \Z} \sum_{j_2 \in \Z}\tilde{P}\left(\hat{\phi}_n(j_1),\hat{\phi}_n(j_2)\right)e^{-i(k-(j_1+j_2))s}dse^{ikt} \\
&\stackrel{\eqref{key}}{=}\sum_{k \notin A_n(N_n)}y^* \sum_{j_1 \in \Z} \sum_{j_2 \in \Z}\tilde{P}\left(\hat{\phi}_n(j_1),\hat{\phi}_n(j_2)\right){1 \over 2\pi}\int_0^{2\pi}e^{-i(k-(j_1+j_2))s}dse^{ikt} \\
&=\sum_{k \notin A_n(N_n)} \sum_{j_1 \in \Z}y^*\tilde{P}\left(\hat{\phi}_n(j_1),\hat{\phi}_n(k-j_1) \right)e^{ikt} \\
&\stackrel{\eqref{key}}{=}\sum_{j_1 \in \Z} y^*\tilde{P}\left(\hat{\phi}_n(j_1),\sum_{k \notin A_n(N_n)}e^{ikt}\hat{\phi}_n(k-j_1) \right) \\
&=\sum_{j_1 \in \Z}y^*\tilde{P}\left(\hat{\phi}_n(j_1),\sum_{k \notin A_n(N_n)}\mathfrak{F}(e^{ij_1 \cdot}\phi_n)(k)e^{ikt} \right) \\
&=\sum_{j_1 \in \Z}y^*\tilde{P}\left(\hat{\phi}_n(j_1),e^{ij_1t}\psi_n(e^{it}) \right) \\
&=y^*\tilde{P}\left(\sum_{j_1 \in \Z}\hat{\phi}_n(j_1)e^{ij_1t},\psi_n(e^{it}) \right) \\
&=y^*\tilde{P}\big(\phi_n(e^{it}),\psi_n(e^{it})\big).
\end{aligned}
$$
Apply finally Lemmas \ref{lem_polar} and \ref{lem2} to conclude the proof.

\end{proof}

\begin{corollary} For $\X, \, \Y$ and $P$ as in Theorem \ref{theo} and for $\left< X_0,X_1\right>_{\theta}$, the Banach space that appears when applying Peetre's Interpolation method, we have that $P:\left< X_0,X_1\right>_{\theta} \rightarrow \mathbf{Y}_{\theta}$ is compact for every $0<\theta<1$.

\end{corollary}

\begin{proof} First of all, notice that $\left< X_0,X_1\right>_{\theta}$ is contained in $\X_{\theta}$ (as pointed out in \cite{Jan,Pe}) and hence we may use all the previous results. More specifically, observe that if $x \in \left< X_0,X_1\right>_{\theta}$, then we can write $x=\sum_{k \in \Z}x_k$, with the series converging unconditionally and therefore
$$
\lim_{n \rightarrow \infty} \sup_{|\lambda_k|\leq 1} \bigg\| \sum_{|k|\geq n}\lambda_k e^{(j-\theta)k}x_k\bigg\|_{X_j}=0
$$
for $j=0, \, 1$. Hence, $\phi(z):=\sum_{k=1}^{\infty}e^{-\theta k}x_kz^k$ is holomorphic on $1 < |z|<e$ and continuous on the boundary (because the series converges uniformly), therefore is an element of $\F(\X)$. \\
Therefore, $B_{\left< X_0,X_1\right>_{\theta}} \subseteq \mho_{\theta}$ and result follows.

\end{proof}

Keeping in mind that compactness of a holomorphic function $f$ and compactness of each of the polynomials that appear in the Taylor series representation of $f$ are equivalent (\cite{ArScho}), we also have the following corollary:

\begin{corollary} \label{cor1}
Let $\mathbf{X}=(X_0,X_1)$ and $\mathbf{Y}=(Y_0,Y_1)$ be compatible couples of Banach spaces and let $f:(X_0,X_1) \rightarrow (Y_0,Y_1)$ so that $f:X_0+X_1 \rightarrow Y_0+Y_1$ and $f:X_j \rightarrow Y_j$ ($j=0,1$) are holomorphic. Assume furthermore that $f|_{X_0}:X_0 \rightarrow Y_0$ is compact. \\
Then, $f:\left<X_0,X_1\right>_{\theta} \rightarrow [Y_0,Y_1]_{\theta}$ is compact, for every $0<\theta<1$.
\end{corollary}

\begin{proof}
We remark first that, by means of Lemma \ref{int_anal}, we obtain that $f:[X_0,X_1]_{\theta} \rightarrow [Y_0,Y_1]_{\theta}$ is holomorphic, so, applying once more the fact that $\left<X_0,X_1\right>_{\theta}$ is contained in $[X_0,X_1]_{\theta}$, we deduce that $f:\left<X_0,X_1\right> \rightarrow [Y_0,Y_1]_{\theta}$ is holomorphic as well. \\
Furthermore, a look at the details in the proof of Lemma \ref{int_anal} shows that the sequence of polynomials that gives holomorphy at one point $x$ is $\left\{{d^mf(x) \over m!}\right\}_{m=0}^{\infty}$. Applying corollary \ref{cor1}, we obtain that
$$
{d^mf(x) \over m!}:\left<X_0,X_1\right>_{\theta} \rightarrow [Y_0,Y_1]_{\theta}
$$
is compact for every $m \geq 0$, and therefore $f:\left<X_0,X_1\right>_{\theta} \rightarrow [Y_0,Y_1]_{\theta}$ is compact as well.
\end{proof}

%%%%%%%%%%%%%%%%%%%%%%%%%%%%%%%%%%%%%%%%%%%%%%%%%%%%%%%%%%%%%%%%%%%%%%%%%%%%%%%%%%%%%%%%%%%%%%%%%
%%%%%%%%%%%%%%%%%%%%%%%%%%%%%%%%%%%%%%%%%%%%%%%%%%%%%%%%%%%%%%%%%%%%%%%%%%%%%%%%%%%%%%%%%%%%%%%%%
\section{Generalization of classic interpolation results.} \label{classic}
%%%%%%%%%%%%%%%%%%%%%%%%%%%%%%%%%%%%%%%%%%%%%%%%%%%%%%%%%%%%%%%%%%%%%%%%%%%%%%%%%%%%%%%%%%%%%%%%%
%%%%%%%%%%%%%%%%%%%%%%%%%%%%%%%%%%%%%%%%%%%%%%%%%%%%%%%%%%%%%%%%%%%%%%%%%%%%%%%%%%%%%%%%%%%%%%%%%

\begin{theorem} \label{main}
Let $\mathbf{X}=(X_0,X_1)$ and $\mathbf{Y}=(Y_0,Y_1)$ be compatible couples of Banach spaces and assume that we can find $\{y_n\}_{n=1}^{\infty} \subseteq Y_0 \cap Y_1$ so that $\{y_n\}_{n=1}^{\infty}$ is a Schauder basis of $(Y_0 \cap Y_1, \| \cdot \|_{Y_j})$ for both $j=0$ and $j=1$. Assume furthermore that $P\in \mathcal{P}(^k\mathbf{X};\mathbf{Y})$ is a continuous homogeneous polynomial for some $k \in \N$ and that $P:X_0 \rightarrow Y_0$ is compact. \\
Then, for every $0\leq \theta \leq 1$, $P:[X_0,X_1]_{\theta} \rightarrow [Y_0,Y_1]_{\theta}$ is compact.
\end{theorem}

\begin{proof}
First of all, notice that we may assume that $Y_0 \cap Y_1$ is dense in $Y_0$ and $Y_1$. Indeed, otherwise we may center our attention on $P:([X_0,X_1]_{0},[X_0,X_1]_{1}) \rightarrow ([Y_0,Y_1]_{0},[Y_0,Y_1]_{1})$, which (as before) is well defined since $P$ is a homogeneous continuous polynomial and, therefore, $P \circ \phi \in \F\{Y_0,Y_1\}$ for every $\phi \in \F\{X_0,X_1\}$. \\
Using Theorem \ref{Ber}, we would have that indeed $Y_0 \cap Y_1$ is dense in both $[Y_0,Y_1]_{0}$ and $[Y_0,Y_1]_{1}$ and that $[Y_0,Y_1]_{\theta}=\Big[[Y_0,Y_1]_{0},[Y_0,Y_1]_{1}\Big]_{\theta}$ (so we would be dealing with the same interpolated spaces). Notice that $[Y_0,Y_1]_{0}$ is a closed subspace of $Y_0$ and that the norm in $[Y_0,Y_1]_{0}$ is the same as $\| \cdot \|_{Y_0}$ and that we would still have that $P:[X_0,X_1]_{0} \rightarrow [Y_0,Y_1]_{0}$ is compact. Define, for every $n \in \N$, $\pi_n:Y_0 \cap Y_1 \rightarrow Y_0 \cap Y_1$ as $\pi_n(\sum_{i=1}^{\infty}a_iy_i)=\sum_{i=1}^na_iy_i$. Then, we know the following:
\begin{itemize}
\item $\pi_n$ is a finite rank operator and therefore it is compact. Due to the fact that $\{y_i\}_{i=1}^{\infty}$ is a Schauder basis, we deduce that $\pi_n$ is continuous. 
\item For every $y \in Y_0 \cap Y_1$ and $j=0,1$, we have $\|y-\pi_ny\|_{Y_j} \xrightarrow[n \rightarrow \infty]\,0$. In particular, $\{y_n\}_{n=1}^{\infty}$ is also a Schauder basis for $(Y_0\cap Y_1, \|\cdot \|_{Y_0 \cap Y_1})$.
\item There exists, for $j=0,1$, $K_j \geq 1$ so that $\|\pi_n\|_{j} \leq K_j$ for every $n \in \N$.
\end{itemize}
We remark that $\pi_n$ admits a continuous extension to $Y_0$ and $Y_1$, maintaining the norm, and therefore it is possible to extend such operators to the whole $Y_0+Y_1$ via $\pi_n(x_0+x_1)=\pi_n(x_0)+\pi_n(x_1)$. In particular, via the Riesz-Thorin theorem, $\pi_n$ is a bounded operator when defined over $[Y_0,Y_1]_{\theta}$. Let us stress that they are still compact operators since they are of finite range. \\
Let us show that $\pi_nP\xrightarrow[n \rightarrow \infty]{\|\cdot\|_{\theta}} P$, so then $P$ would be the limit of compact operators, and hence compact. Indeed, notice first
$$
\begin{aligned}
\|P-\pi_nP\|_{\theta}&\leq \|\tilde{P}-\widetilde{\pi_nP}\|_{\theta} \leq \|\tilde{P}-\widetilde{\pi_nP}\|^{1-\theta}_{0}\|\tilde{P}-\widetilde{\pi_nP}\|^{\theta}_{1} \leq {k^k \over k!} \|P-\pi_nP\|^{1-\theta}_{0}\|P-\pi_nP\|^{\theta}_{1} \\
&\leq {k^k \over k!}\|P-\pi_nP\|^{1-\theta}_{0}\Big(\|P\|_{1}\|I-\pi_n\|_{1}\Big)^{\theta} \leq \Big[(1+K_1)\|P\|_{1}\Big]^{\theta}{k^k \over k!}\|P-\pi_nP\|^{1-\theta}_{0}.
\end{aligned}
$$
Let now $\epsilon>0$. By compactness, we know that we can find $x_1, \ldots,x_m \in \overline{B}_{X_0}$ so that 
$$
P(\overline{B}_{X_0}(0;1)) \subseteq \bigcup_{i=1}^m\overline{B}_{Y_0}(P(x_i);\epsilon).
$$
We can also find $n_{\epsilon} \in \N$ so that, for every $n \geq n_{\epsilon}$ and $j=1, \ldots, m$,
$$
\|Px_j-\pi_nPx_j\|_{Y_0} \leq \epsilon.
$$
Therefore, if $n \geq n_{\epsilon}$ and $x \in \overline{B}_{X_0}$, we can choose $1 \leq j_0 \leq m$ so that 
$$
Px \in \overline{B}_{Y_0}(P(x_{j_0});\epsilon)
$$ 
and set
$$
\begin{aligned}
\|Px-\pi_nPx\|_{Y_0} & \leq \|Px-Px_{j_0}\|_{Y_0}+\|Px_{j_0}-\pi_nPx_{j_0}\|_{Y_0}+\|\pi_nPx_{j_0}-\pi_nPx\|_{Y_0} \\
& \leq (1+K_0)\|Px-Px_{j_0}\|_{Y_0}+\epsilon<(2+K_0)\epsilon.
\end{aligned}
$$
Hence,
$$
\sup_{x\in \overline{B}_{X_0}}\|Px-\pi_nPx\|_{Y_0}=\|P-\pi_nP\|_{0} \leq (2+K)\epsilon.
$$
and we can then conclude, for every $n \geq n_{\epsilon}$, that
$$
\|P-\pi_nP\|_{\theta}\leq \Big[(1+K_1)\|P\|_{1}\Big]^{\theta}{k^k \over k!}[(2+K_0)\epsilon]^{1-\theta} \xrightarrow[\epsilon \rightarrow 0]\,0.
$$

\end{proof}

\begin{corollary}
Let $(X_0,X_1)$ be a compatible couple of Banach spaces and $Y_0, \, Y_1$ be either $L^{q_0}(K)$ or $L^{q_1}(K)$ (with $K \subseteq \R^n$ a compact set and $1 \leq q_0, \, q_1 \leq \infty$) or $\ell^{q_0}, \, \ell^{q_1}$ ($1 \leq q_0, \, q_1<\infty$). Assume $k \in \N$ and $P\in\mathcal{P}(^k(X_0,X_1),(Y_0,Y_1))$ so that $P:X_0 \rightarrow Y_0$ is compact. \\
Then, $P:[X_0,X_1]_{\theta} \rightarrow [Y_0,Y_1]_{\theta}$ is compact for every $0 \leq \theta \leq 1$.
\end{corollary}

\begin{proof}
Just notice that $[Y_0,Y_1]_{\theta} \in \{L^q(K), \, \ell^q \}$ with ${1 \over q}={1-\theta \over q_0}+{\theta \over q_1}$ and those spaces admit a Schauder basis which is common for all of the interpolated spaces (the dual to the coordinate operators, $\{e_n\}_{n=1}^{\infty}$, for the $\ell^p$ spaces and the Haar system for the $L^p(K)$ spaces).
\end{proof}

In 1957, Pe\l czynski showed that if $P:\ell_p \rightarrow \ell_q$ is a bounded homogeneous polynomial of degree $n$, then $P$ is compact, provided $nq<p$ (\cite{Pel}). Keeping that in mind, we have the following result:
\begin{corollary}
Let $1 \leq p,q <\infty$ and $P:\ell_p \rightarrow \ell_q$ be a homogeneous bounded polynomial of degree $n$ which is not compact. Then, if there exists $1 \leq r \leq p$ and $\epsilon>0$ so that $P(\ell_r) \subseteq \ell_{{r \over n}-\epsilon}$, then $P:(\ell_r,\| \cdot \|_r) \rightarrow (\ell_{{r \over n}-\epsilon}, \| \cdot \|_{{r \over n}-\epsilon})$ is not bounded.
\end{corollary} 

\begin{proof}
Indeed, otherwise we would have that $P:(\ell_r,\| \cdot \|_r) \rightarrow (\ell_{{r \over n}-\epsilon}, \| \cdot \|_{{r \over n}-\epsilon})$ is compact, applying the result by Pe\l czynski. If we consider now
$$
P:(\ell_p,\ell_r) \rightarrow (\ell_q,\ell_{{r \over n}-\epsilon}) 
$$
we would have then that, for every $0 \leq \theta \leq 1$, $P:[\ell_r, \ell_p]_{\theta} \rightarrow [\ell_q,\ell_{{r \over n}-\epsilon}]_{\theta}$ is compact. In particular, taking $\theta=0$, we have that $[\ell_p,\ell_r]_{0}=\ell_p$ and $[\ell_q,\ell_{{r \over n}-\epsilon}]_{0}=\ell_q$, so $P:\ell_p \rightarrow \ell_q$ would be compact, reaching hence a contradiction.
\end{proof}

The following theorem generalizes a result presented by Lions and Peetre in \cite{LiPe} (Theorems 2.1 and 2.2 from ch. V, pp. 36--37).

\begin{theorem} \label{later}
Let $(X_0,X_1)$ and $(Y_0,Y_1)$ be two compatible couples of Banach spaces, $X, \, Y$ be Banach spaces and let $0<\theta<1$. 
\begin{enumerate}
\item \label{p1} If $P \in \P(^mX,(Y_0,Y_1))$ and $P:X \rightarrow Y_0$ is compact, then $P:X \rightarrow \mathbf{Y}_{\theta}$ is compact.
\item \label{p2} If $P \in \P(^m(X_0,X_1),Y)$ and $P:X_0 \rightarrow Y$ is compact, then $P:\mathbf{X}_{\theta} \rightarrow Y$ is compact.
\end{enumerate}
\end{theorem}

Before giving the details of the proof, we will need to state the following technical lemma:
\begin{lemma} \label{technical}
Let $X, \, Y$ be two Banach spaces and let $T:\stackrel{m}{\overbrace{X \times \ldots \times X}} \rightarrow Y$ be a symmetric multilinear operator. Then, for every $x_0, x_1 \in X$, we have
$$
T(x_0+x_1, \ldots,x_0+x_1)=T(x_0, \ldots,x_0)-\sum_{k=1}^m(-1)^k {m \choose k}T(\stackrel{m-k}{\overbrace{x_0+x_1, \ldots,x_0+x_1}},\stackrel{k}{\overbrace{x_1, \ldots,x_1}})
$$
\end{lemma}

\begin{proof}
We will proceed via induction on $m$. For $m=1$, the result is trivial, since then the claim is just the linearity of $T$. \\
Assume the result is true for $m$. Then,
$$
\begin{aligned}
T(\stackrel{m+1}{\overbrace{x_0+x_1, \ldots,x_0+x_1}})&=T(x_0,\stackrel{m}{\overbrace{x_0+x_1, \ldots,x_0+x_1}})+T(x_1,\stackrel{m}{\overbrace{x_0+x_1, \ldots,x_0+x_1}}) \\
&=T(x_0, \ldots,x_0)-\sum_{k=1}^m(-1)^k{m \choose k}T(x_0,\stackrel{m-k}{\overbrace{x_0+x_1, \ldots,x_0+x_1}},\stackrel{k}{\overbrace{x_1, \ldots,x_1}}) \\
&\quad+T(x_1,x_0, \ldots,x_0)-\sum_{k=1}^m(-1)^k{m \choose k}T(x_1,\stackrel{m-k}{\overbrace{x_0+x_1, \ldots,x_0+x_1}},\stackrel{k}{\overbrace{x_1, \ldots,x_1}}) \\
&=T(x_0, \ldots,x_0)-\sum_{k=1}^m(-1)^k{m \choose k}T(\stackrel{m-k+1}{\overbrace{x_0+x_1, \ldots,x_0+x_1}},\stackrel{k}{\overbrace{x_1, \ldots,x_1}}) \\
& \quad+T(x_1,x_0,\ldots,x_0) \\
&=T(x_0, \ldots,x_0)-\sum_{k=1}^m(-1)^k{m \choose k}T(\stackrel{m-k+1}{\overbrace{x_0+x_1, \ldots,x_0+x_1}},\stackrel{k}{\overbrace{x_1, \ldots,x_1}}) \\
& \quad+T(x_1,x_0+x_1,\ldots,x_0+x_1) \\
& \quad+\sum_{k=1}^m(-1)^k{m \choose k}T(\stackrel{m-k}{\overbrace{x_0+x_1, \ldots,x_0+x_1}},\stackrel{k+1}{\overbrace{x_1, \ldots,x_1}}) \\
&=T(x_0, \ldots,x_0)+(m+1)T(x_1,x_0+x_1,\ldots,x_0+x_1) \\
& \quad-\sum_{k=2}^{m}(-1)^k{m \choose k}T(\stackrel{m-k+1}{\overbrace{x_0+x_1, \ldots,x_0+x_1}},\stackrel{k}{\overbrace{x_1, \ldots,x_1}}) \\
& \quad - \sum_{k=2}^{m+1}(-1)^k{m \choose k-1}T(\stackrel{m-k+1}{\overbrace{x_0+x_1, \ldots,x_0+x_1}},\stackrel{k}{\overbrace{x_1, \ldots,x_1}}) \\
&=T(x_0, \ldots,x_0)+(m+1)T(x_1,x_0+x_1,\ldots,x_0+x_1) \\
& \quad-\sum_{k=2}^m(-1)^k\left[{m \choose k}+{m \choose k-1}\right]T(\stackrel{m-k+1}{\overbrace{x_0+x_1, \ldots,x_0+x_1}},\stackrel{k}{\overbrace{x_1, \ldots,x_1}}) \\
& \quad-(-1)^{m+1}T(\stackrel{m+1}{\overbrace{x_1,\ldots,x_1}}) \\
&=T(x_0, \ldots,x_0)-\sum_{k=1}^{m+1}(-1)^k {m+1 \choose k}T(\stackrel{m+1-k}{\overbrace{x_0+x_1, \ldots,x_0+x_1}},\stackrel{k}{\overbrace{x_1, \ldots,x_1}}).
\end{aligned}
$$

\end{proof}

\newpage

\begin{proof}[Proof of Theorem \ref{later}] \textcolor{white}{hola} \vspace{30pt}
\begin{enumerate}
\item Let $\{x_n\}_{n=1}^{\infty} \subseteq X$ be a bounded sequence. Then, we can find a subsequence $\{x_{n_k}\}_{k=1}^{\infty}$ so that $\{Px_{n_k}\}_{k=1}^{\infty} \subseteq Y_0 \cap Y_1$ is a Cauchy sequence with respect to $\|\cdot \|_{Y_0}$. Now, notice that 
$$
\begin{aligned}
\|Px_{n_k}-Px_{n_l}\|_{\theta} & \leq C \|Px_{n_k}-Px_{n_l}\|_0^{1-\theta}\|Px_{n_k}-Px_{n_l}\|_1^{\theta} \\
& \leq C\|Px_{n_k}-Px_{n_l}\|_0^{1-\theta}\left(2\|P\|_{X \rightarrow Y_1}\sup_n \|x_n\|_X^m \right)^{\theta} \xrightarrow[k,l \rightarrow \infty]\,0.
\end{aligned}
$$
\item Let us show that $P(B_{\mathbf{X}_{\theta}})$ is a relatively compact subset of $Y$. Indeed, let $\epsilon>0$. Let $t\geq 1$ so that 
$$
\|\tilde{P}\|_{(X_0+X_1) \rightarrow Y}C't^{\theta-1}\sum_{k=1}^m{m \choose k}<{\epsilon \over 2},
$$
where $m$ is the degree of $P$ and $C'>0$ is the constant given by Theorem \ref{Lions}, part \ref{part2}. Apply next relative compactness of $P(B_{X_{0}})$ to find elements $x^{(1)}, \ldots, x^{(n)} \in B_{X_0}$ so that 
$$
P(B_{X_0}) \subseteq \cup_{j=1}^nB(P(x^{(j)});{\epsilon \over 2}).
$$
Then, if $x \in B_{\mathbf{X}_{\theta}}$, we can apply the Theorem \ref{Lions}, part \ref{part2}, to obtain a decomposition $x=x_0+x_1$ with $x_0 \in X_0$, $x_1 \in X_1$, $\|x_0 \|_{X_0} \leq C't^{\theta}$ and $\|x_1\|_{X_1} \leq C't^{\theta-1}$. Choose also $1 \leq j_0 \leq n$ so that $\|P(x_0)-P(x^{(j_0)})\|_{Y} <{\epsilon \over 2}$. Then, using Lemma \ref{technical}
$$
\begin{aligned}
\|P(x)-P(x^{(j_0)})\|_{Y} &=\|P(x_0+x_1)-P(x^{(j_0)})\|_Y \\
&=\Big\|P(x_0)-\sum_{k=1}^m(-1)^k {m \choose k}\tilde{P}(\stackrel{m-k}{\overbrace{x_0+x_1, \ldots,x_0+x_1}},\stackrel{k}{\overbrace{x_1, \ldots,x_1}})-P(x^{(j_0)})\Big\|_Y \\
& \leq {\epsilon \over 2}+\|\tilde{P}\|_{X_0+X_1 \rightarrow Y}\sum_{k=1}^m{m \choose k}\|x_0+x_1\|^{m-k}_{X_0+X_1}\|x_1\|_{X_0+X_1}^{k} \\
& \leq {\epsilon \over 2}+\|\tilde{P}\|_{X_0+X_1 \rightarrow Y}\sum_{k=1}^m{m \choose k}\|x_0+x_1\|^{m-k}_{[X_0,X_1]_{\theta}}\|x_1\|_{X_1}^{k} \\
& \leq {\epsilon \over 2}+\|\tilde{P}\|_{X_0+X_1 \rightarrow Y}\sum_{k=1}^m{m \choose k}C't^{k(\theta-1)} \leq {\epsilon \over 2}+\|\tilde{P}\|_{X_0+X_1 \rightarrow Y}C't^{\theta-1}\sum_{k=1}^m{m \choose k} \\
& <\epsilon.
\end{aligned}
$$
Therefore, we can conclude that $P(B_{[X_0,X_1]_{\theta}}) \subseteq \cup_{j=1}^n B(P(x^{(j_0)});\epsilon)$.
\end{enumerate}
\end{proof}

\begin{theorem}
Let $\mathbf{X}=(X_0,X_1), \, \mathbf{Y}=(Y_0,Y_1)$ be two compatible couples of Banach spaces and let $P \in \P(^n\mathbf{X};\mathbf{Y})$ be a bounded homogeneous polynomial, so that $P_{|X_0}:X_0 \rightarrow Y_0$ is compact. Assume that we can find a family of polynomials $\{P_{\lambda}:Y_0+Y_1 \rightarrow Y_0 \cap Y_1 \}_{\lambda \in \Lambda}$ and a constant $C>0$ so that $\|P_{\lambda}\|_{Y_j,Y_j} \leq C$ (for $j=0,1$ and $\lambda \in \Lambda$) and, for every $\epsilon>0$ we can find $\lambda_{\epsilon} \in \Lambda$ so that $\|Px-P_{\lambda_{\epsilon}}x\|_{Y_0}<\epsilon$ for every $x \in B_{[X_0,X_1]_{\theta}}$. \\
Then, $P:[X_0,X_1]_{\theta} \rightarrow [Y_0,Y_1]_{\theta}$ is compact.
\end{theorem}

\begin{proof}
Given $\lambda \in \Lambda$, define the homogeneous polynomial $Q_{\lambda}=P_{\lambda} \circ P: X_0+X_1 \rightarrow Y_0 \cap Y_1 \hookrightarrow [Y_0,Y_1]_{\theta}$. Then, in particular, $Q_{\lambda}:X_0 \rightarrow [Y_0,Y_1]_{\theta}$ is compact and, applying Theorem \ref{later}, part \ref{p2}, we get that $Q_{\lambda}:[X_0,X_1]_{\theta} \rightarrow [Y_0,Y_1]_{\theta}$ is compact. \\
Let us show that we can approximate $P$ by $\{Q_{\lambda}\}_{\lambda \in \Lambda}$, in the uniform norm. Indeed, let $\epsilon>0$ and let $C>0$ be the constant given by the hypothesis. We can then find $\lambda_0 \in \Lambda$ so that
$$
\|Px-P_{\lambda_0}Px\|_{Y_0} \leq \epsilon,
$$
for every $x \in B_{X_0}$. Then, using Theorem \ref{multilinear}
$$
\begin{aligned}
\|P-Q_{\lambda_0}\|_{[X_0,X_1]_{\theta} \rightarrow [Y_0,Y_1]_{\theta}} & \leq {n^n \over n!}\|P-Q_{\lambda_0}\|_{X_0 \rightarrow Y_0}^{1-\theta}\|P-Q_{\lambda_0}\|_{X_1 \rightarrow Y_1}^{\theta}<{n^n \over n!} \epsilon^{1-\theta}(1+C)^{\theta}\|P\|_{X_1 \rightarrow Y_1}.
\end{aligned}
$$
\end{proof}

{\bf Acknowledgements} The present paper was completed while the author was completing his Ph.D. in Kent State University. The author would also like to thank Professor M. Cwikel for his selfless help in filling the details of his results in \cite{CwiKal}. The author is specially grateful to Professor Richard M. Aron for proposing this topic to him and his advice and guidance throughout its contents.

%%%%%%%%%%%%%%%%%%%%%%%%%%%%%%%%%%%%%%%%%%%%%%%%%%%%%%%%%%%%%%%%%%%%%
%%%%%%%%%%%%%%%%%%%%%%%%%%%%%%%%%%%%%%%%%%%%%%%%%%%%%%%%%%%%%%%%%%%%%


\begin{thebibliography}{99}
%%%%%%%%%%%%%%%%%%%%%%%%%%%%%%%%%%%%%%%%%%%%%%%%%%%%%%%%%%%%%%%%%%%%%
%%%%%%%%%%%%%%%%%%%%%%%%%%%%%%%%%%%%%%%%%%%%%%%%%%%%%%%%%%%%%%%%%%%%%

\bibitem{ArMo} M.G. Armentano, V. Moreno, {\it Interpolation in Jacobi-weighted spaces and its application to a posteriori error estimations of the $p$-version of the finite element method}, Appl. Numer. Math. {\bf 109} (2016), 184--207.

\bibitem{ArScho}R.M. Aron, M. Schottenloher, {\it Compact holomorphic mappings on Banach spaces and the approximation property}, J. Funct. Anal., no. {\bf 21} (1976), 7--30.

\bibitem{AsKruMas} I. Asekritova, N. Kruglyak, M. Masty{\l}o, {\it Interpolation of Fredholm operators}, Adv. Math., {\bf 295} (2016), 421--496.

\bibitem{Be} J. Bergh, J.L\"ofstr\"om, {\it Interpolation spaces. An introduction}, Springer, Berlin, 1976.

\bibitem{Calderon} A.P. Calder\'on, {\it Intermediate spaces and interpolation, the complex method}, Studia Math., {\bf 24} (1964), 113--190.

\bibitem{CamFerManMaNa} R. del Campo, A. Fern\'andez, A. Manzano, F. Mayoral, F. Naranjo, {\it Interpolation with a parameter function of $L_p$-spaces with respect to a vector measure on a $\delta$-ring}, Banach J. Math. Anal., {\bf 10} (2016), no. 4, 815--827.

\bibitem{CoDo} F. Cobos, \'O. Dom\'inguez, {\it On the relationship between two kinds of Besov spaces with smoothness near zero and some other applications of limiting interpolation}, J. Fourier Anal. Appl., {\bf 22} (2016), no. 5, 1174--1191.

\bibitem{CoEdPo} F. Cobos, D.E. Edmunds, A.J.B. Potter, {\it Real interpolation and compact linear operators}, J. Funct. Anal., {\bf 88} (1990), 351--365.

\bibitem{CoFerMar} F. Cobos, L.M. Fern\'andez-Cabrera, A. Mart\'inez, {\it Complex interpolation, minimal methods and compact operators}, Math. Nachr., {\bf 263-264} (2004), 67--82.

\bibitem{CoPe} F. Cobos, J. Peetre, {\it Interpolation of compactness using Aronszajn-Gagliardo functors}, Israel J. Math., {\bf 68} (1989), 220--240.

\bibitem{CoSe} F. Cobos, A. Segurado, {\it Description of logarithmic interpolation spaces by means of the $J$-functional and applications}, J. Funct. Anal., {\bf 268} (2015), no. 10, 2906--2945.

\bibitem{Cwi} M. Cwikel, {\it Real and complex interpolation and extrapolation of compact operators}, Duke Math. J., {\bf 27} (1992), 1005--1009.

\bibitem{Cwi2} M. Cwikel, {\it Complex interpolation of compact operators mapping into lattice couples}, Proc. Est. Acad. Sci., {\bf 59} (2010), no. 1, 19--28.

\bibitem{CwiKal} M. Cwikel, N.J. Kalton, {\it Interpolation of compact operators by the methods of Calder\'on and Gustavsson-Peetre}, Proc. Ed. Math. Soc., {\bf 38} (1995), 261--276.

\bibitem{CwiKrugMas} M. Cwikel, N. Krugljak, M. Mastylo, {\it On compact interpolation of compact operators}, Illinois J. Math., {\bf 40} (1996), no. 3, 353--364.

\bibitem{Di} S. Dineen, Complex analysis on infinite-dimensional spaces, Springer Monographs in Mathematics, Springer-Verlag London, Ltd., London, 1999, ISBN 1-85233-158-5.

\bibitem{Fab} M. Fabian, P. Habala, P. H\'ajek, V. Montesinos Santaluc\'ia, J. Pelant, V. Zizler, {\it Functional analysis and infinite-dimensional geometry}, Canadian Mathematical Society, ISBN 0-387-95219-5.

\bibitem{FerSil} D.L. Fern\'andez, E.B. da Silva, {\it Interpolation of bilinear operators and compactness}, Nonlinear Anal., {\bf 73} (2010), 526--537.

\bibitem{Fer} L.M. Fern\'andez-Cabrera, {\it The fundamental function of spaces generated by interpolation methods associated to polygons}, Mediterr. J. Math., {\bf 14} (2017), no. 1, 14--17.

\bibitem{FerMar} L.M. Fern\'andez-Cabrera, A. Mart\'inez, {\it On interpolation properties of compact bilinear operators}, Math. Nachr., {\it accepted for publication}.

\bibitem{Jan} S. Janson, {\it Minimal and maximal methods of interpolation}, J. Funct. Anal., {\bf 44} (1981), 50--73.

\bibitem{HaRo} M. Haas, J. Rozendaal, {\it Functional calculus on real interpolation spaces for generators of $C_0$-groups}, Math. Nachr., {\bf 289} (2016), no. 2-3, 275--289.

\bibitem{Kras} M. A. Krasnosel'skii, {\it On a theorem of M. Riesz}, Soviet Math. Dokl., {\bf 1} (1960), 229--231.

\bibitem{LiPe} J.L. Lions, J. Peetre, {\it Sur une classe d'espaces d'interpolation}, Inst. Hautes \'Etudes Sci. Publ. Math., {\bf 19} (1964), 5--68.

\bibitem{Mar} R.S. Martin, Ph. D. Thesis, Cal. Inst. of Tech., 1932.

\bibitem{MasSin} M. Masty{\l}o, G. Sinnamon, {\it Calder\'on-Mityagin couples of Banach spaces related to decreasing functions}, J. Funct. Anal., {\bf 272} (2017), no. 11, 4460--4482.

\bibitem{Mu} J. Mujica, {\it Holomorphic functions on Banach spaces}, Note di Matematica, {\bf 25} (2005), no. 2, 113--138.

\bibitem{Pe} J. Peetre, {\it Sur l'utilization des suites inconditionallement sommables dans la th\'eorie des espaces d'interpolation}, Rend. Sem. Mat. Univ. Padova, {\bf 46} (1971), 173--190.

\bibitem{Pel} A. Pe\l czynski, {\it A property of multilinear operations}, Studia Math. {\bf 16} (1957--1958), 173--182.

\bibitem{Qiu} Y. Qiu, {\it On the effect of rearrangement on complex interpolation for families of Banach spaces}, Rev. Mat. Iberoam., {\bf 31} (2015), no. 2, 439--460.

\bibitem{Ri} M. Riesz, {\it Sur les maxima des formes bilinéaires et sur les fonctionnelles linéaires}, Acta Mathematica, {\bf 49} (1927) 465--497.

\bibitem{Tho} G.O. Thorin, {\it Convexity theorems generalizing those of M. Riesz and Hadamard with some applications}, Comm. Sem. Math. Univ. Lund {\bf 9} (1948), 1--58.

\end{thebibliography}
\end{document}